\documentclass[oneside,english]{amsart}[11]
\usepackage[T1]{fontenc}
\usepackage[latin9]{inputenc}
\usepackage{geometry}
\usepackage{amsbsy}
\usepackage{amstext}
\usepackage{amsthm}
\usepackage{amssymb, enumerate}
\usepackage{color}

\makeatletter
\numberwithin{equation}{section}
\numberwithin{figure}{section}
\theoremstyle{plain}
\newtheorem{thm}{\protect\theoremname}[section]
  \theoremstyle{plain}
  
  \theoremstyle{definition}
  
  \theoremstyle{definition}
  \newtheorem{defn}[thm]{\protect\definitionname}
  \theoremstyle{plain}
  \newtheorem{prop}[thm]{\protect\propositionname}
  \theoremstyle{plain}
  \newtheorem{cor}[thm]{\protect\corollaryname}
  \theoremstyle{definition}
  \newtheorem{example}[thm]{\protect\examplename}
  \theoremstyle{plain}
  \newtheorem{lem}[thm]{\protect\lemmaname}
  \theoremstyle{definition}
  
  \theoremstyle{definition}
  \newtheorem{setup}[thm]{Setup}

\makeatother

\usepackage{babel}
  \providecommand{\conjecturename}{Conjecture}
  \providecommand{\corollaryname}{Corollary}
  \providecommand{\definitionname}{Definition}
  \providecommand{\examplename}{Example}
  \providecommand{\lemmaname}{Lemma}
  \providecommand{\problemname}{Problem}
  \providecommand{\propositionname}{Proposition}
\providecommand{\theoremname}{Theorem}
\providecommand{\claimname}{Claim}

\usepackage{thmtools}

\begin{document}
\global\long\def\connected{\text{highly connected}}
\global\long\def\tra{\mathsf{T}_{q\to p}}
\global\long\def\art{\mathsf{T}^{p\to q}}
 \global\long\def\f{\mathcal{F}}
\global\long\def\pn{\mathcal{P}\left[n\right]}
\global\long\def\g{\mathcal{G}}
\global\long\def\s{\mathcal{S}}
\global\long\def\j{\mathcal{J}}
\global\long\def\d{\mathcal{D}}
\global\long\def\Inf{}
\global\long\def\p{\mathcal{P}}
\global\long\def\h{\mathcal{H}}
\global\long\def\n{\mathbb{N}}
\global\long\def\a{\mathcal{A}}
\global\long\def\b{\mathcal{B}}
\global\long\def\c{\mathcal{C}}
\global\long\def\e{\mathbb{E}}
\global\long\def\x{\mathbf{x}}
\global\long\def\y{\mathbf{y}}
\global\long\def\z{\mathbf{z}}
\global\long\def\c{\mathbf{c}}
\global\long\def\av{\mathsf{A}}
\global\long\def\chop{\mathrm{Chop}}
\global\long\def\stab{\mathrm{Stab}}
\global\long\def\t{\mathsf{T}}
\global\long\def\fabs#1{\hat{|}#1\hat{|}}
\newcommand{\cal }{\mathcal }

\newcommand\comment[1]{\fbox{\tt #1}}
\newcommand\E{\mathop{\mathbb{E}}}
\newcommand\card[1]{\left| {#1} \right|}
\newcommand\sett[2]{\left\{ \left. #1 \;\right\vert #2 \right\}}
\newcommand\settright[2]{\left\{ #1 \;\left\vert\; #2 \right.\right\}}
\newcommand\set[1]{{\left\{ #1 \right\}}}
\newcommand\Prob[2]{{\Pr_{#1}\left[ {#2} \right]}}
\newcommand\BProb[2]{{\Pr_{#1}\Big[ {#2} \Big]}}
\newcommand\bProb[2]{{\Pr_{#1}\big[ {#2} \big]}}
\newcommand\cProb[3]{{\Pr_{#1}\left[ \left. #3 \;\right\vert #2 \right]}}
\newcommand\cbProb[3]{{\Pr_{#1}\big[ \big. #3 \;\big\vert #2 \big]}}
\newcommand\cBProb[3]{{\Pr_{#1}\Big[ \Big. #3 \;\Big\vert #2 \Big]}}
\newcommand\Expect[2]{{\mathop{\mathbb{E}}_{#1}\left[ {#2} \right]}}
\newcommand\bExpect[2]{{\mathop{\mathbb{E}}_{#1}\big[ {#2} \big]}}
\newcommand\BExpect[2]{{\mathop{\mathbb{E}}_{#1}\Big[ {#2} \Big]}}
\newcommand\cProbright[3]{{\Pr_{#1}\left[ #3\;\left\vert  #2\right. \right]}}
\newcommand\cExpect[3]{{\mathbb{E}_{#1}\left[ \left. #3 \;\right\vert #2 \right]}}
\newcommand\norm[1]{\| #1 \|}
\newcommand\power[1]{\set{0,1}^{#1}}
\newcommand\ceil[1]{\lceil{#1}\rceil}
\newcommand\round[1]{\left\lfloor{#1}\right\rfloor}
\newcommand\half{{1\over2}}
\newcommand\quarter{{1\over4}}
\newcommand\eigth{{1\over8}}
\newcommand\defeq{\stackrel{def}{=}}
\newcommand\skipi{{\vskip 10pt}}
\newcommand\sqrtn{ {\sqrt n} }
\newcommand\cube[1]{ { \set{-1,1}^{#1} } }
\newcommand\minimum[2] {  \min\left\{{{#1}, {#2}}\right\}   }
\newcommand\spn{{\sf Span}}
\newcommand\spnn[1]{{\sf Span}({#1})}
\newcommand\inner[2]{\langle{#1},{#2}\rangle}
\newcommand\eps{\varepsilon}
\newcommand\sgn[1]{sign({#1})}
\newcommand*\xor{\mathbin{\oplus}}
\newcommand*\bxor{\mathbin{\bigoplus}}
\newcommand*\Hell[2]{{H\left({#1}, \ {#2}\right)}}
\newcommand*\cHell[4]{{H\left({#1}\vert({#2}),\ {#3}\vert({#4}) \right)}}
\newcommand*\dimm{{\sf dim}}
\newcommand*\tp{t}
\newcommand*\R{\mathbb{R}}
\newcommand*\mult[1]{\vec{#1}}
\newcommand*\zin[2]{\mu_{{\sf in}({#1})}({#2})}
\newcommand*\zout[2]{\mu_{{\sf out}({#1})}({#2})}
\newcommand*\zinout[3]{\mu_{{\sf in}({#1}), {\sf out}({#2})}({#3})}
\newcommand\lvlfn[2]{{#1}_{={#2}}}
\newcommand\approxlvlfn[2]{{#1}_{\approx{#2}}}
\newcommand{\cupdot}{\mathbin{\mathaccent\cdot\cup}}
\newcommand{\M}{\mathcal{M}}
\newcommand{\indeg[1]}{{\sf in}\text{-}{\sf deg}({#1})}
\newcommand{\outdeg[1]}{{\sf out}\text{-}{\sf deg}({#1})}
\newcommand{\NP}{\sf NP}

\newcommand{\mc}[1]{\mathcal{#1}}
\newcommand{\mb}[1]{\mathbb{#1}}
\newcommand{\mf}[1]{\mathfrak{#1}}
\newcommand{\nib}[1]{\noindent {\bf #1}}
\newcommand{\nim}[1]{\noindent {\em #1}}
\newcommand{\brac}[1]{\left( #1 \right)}
\newcommand{\bracc}[1]{\left\{ #1 \right\}}
\newcommand{\bsize}[1]{\left| #1 \right|}
\newcommand{\brak}[1]{\left[ #1 \right]}
\newcommand{\bgen}[1]{\left\langle #1 \right\rangle}
\newcommand{\sgen}[1]{\langle #1 \rangle}
\newcommand{\bfl}[1]{\left\lfloor #1 \right\rfloor}
\newcommand{\bcl}[1]{\left\lceil #1 \right\rceil}

\newcommand{\sub}{\subset}
\newcommand{\subn}{\subsetneq}
\newcommand{\sups}{\supseteq}
\newcommand{\ex}{\mbox{ex}}
\newcommand{\nsub}{\notsubseteq}

\newcommand{\lra}{\leftrightarrow}
\newcommand{\Lra}{\Leftrightarrow}
\newcommand{\Ra}{\Rightarrow}
\newcommand{\sm}{\setminus}
\newcommand{\ov}{\overline}
\newcommand{\ul}{\underline}
\newcommand{\wt}{\widetilde}
\newcommand{\es}{\emptyset}
\newcommand{\pl}{\partial}
\newcommand{\wg}{\wedge}

\newcommand{\la}{\leftarrow}
\newcommand{\ra}{\rightarrow}
\newcommand{\da}{\downarrow}
\newcommand{\ua}{\uparrow}

\newcommand{\aA}{\alpha}
\newcommand{\bB}{\beta}
\newcommand{\gG}{\gamma}
\newcommand{\dD}{\delta}
\newcommand{\iI}{\iota}
\newcommand{\kK}{\kappa}
\newcommand{\zZ}{\zeta}
\newcommand{\lL}{\lambda}
\newcommand{\tT}{\theta}
\newcommand{\sS}{\sigma}
\newcommand{\oO}{\omega}
\newcommand{\ups}{\upsilon}

\newcommand{\Tt}{\Theta}
\newcommand{\DD}{\Delta}
\newcommand{\Oo}{\Omega}

\newcommand{\xb}{\boldsymbol{x}}
\newcommand{\yb}{\boldsymbol{y}}
\newcommand{\zb}{\boldsymbol{z}}

\title{Hypercontractivity for global functions and sharp thresholds}

\author{Peter Keevash \qquad Noam Lifshitz \qquad Eoin Long \qquad Dor Minzer}

\thanks{Research supported in part by ERC Consolidator Grant 647678 (PK)
and by NSF grant CCF-1412958 and Rothschild Fellowship (DM)}

\begin{abstract}

The classical hypercontractive inequality for the noise operator
on the discrete cube plays a crucial role in many of the
fundamental results in the Analysis of Boolean functions,
such as the KKL (Kahn-Kalai-Linial) theorem,
Friedgut's junta theorem and the invariance principle
of Mossel, O'Donnell and Oleszkiewicz.
In these results the cube is equipped
with the uniform ($1/2$-biased) measure,
but it is desirable, particularly for applications
to the theory of sharp thresholds, to also obtain
such results for general $p$-biased measures.
However, simple examples show that when $p$ is small
there is no hypercontractive inequality
that is strong enough for such applications.

In this paper, we establish an effective hypercontractivity
inequality for general $p$ that applies to `global functions',
i.e.\ functions that are not significantly
affected by a restriction of a small set of coordinates.
This class of functions appears naturally,
e.g.\ in Bourgain's sharp threshold theorem,
which states that such functions exhibit a sharp threshold.
We demonstrate the power of our tool by
strengthening Bourgain's theorem,
thereby making progress on a conjecture of Kahn and Kalai.
An additional application of our hypercontractivity theorem,
is a $p$-biased analog of the seminal invariance principle
of Mossel, O'Donnell, and Oleszkiewicz. In a companion paper,
we give applications to the solution of two open problems
in Extremal Combinatorics.
\end{abstract}

\maketitle

\section{Introduction}

The field of analysis of Boolean functions is centered around the
study of functions on the discrete cube $\{0,1\}^{n}$,
via their Fourier--Walsh expansion,
often using the classical hypercontractive inequality
for the noise operator, obtained independently
by Bonami \cite{bonami1970etude},
Gross \cite{gross1975logarithmic}
and Beckner \cite{beckner1975inequalities}.
In particular, the fundamental KKL theorem of Kahn, Kalai and Linial \cite{kahn1988influence} applies hypercontractivity to obtain
structural information on Boolean valued functions with
small `total influence' / `edge boundary'
(see Section \ref{subsec:isoinf});
such functions cannot be `global':
they must have a co-ordinate with large influence.

The theory of sharp thresholds is closely connected
(see Section \ref{sec:kahn kalai}) to the structure
of Boolean functions of small total influence,
not only in the KKL setting of uniform measure on the cube,
but also in the general $p$-biased setting.
However, we will see below that the hypercontractivity theorem
is ineffective for small $p$.
This led Friedgut \cite{friedgut1999sharp},
Bourgain \cite[appendix]{friedgut1999sharp}, and Hatami \cite{hatami2012structure} to develop new ideas for proving
$p$-biased analogs of the KKL theorem.
The theme of these works
can be roughly summarised by the statement:\
an effective analog of the KKL theorem holds
for a certain class of `global' functions.
However, these theorems were incomplete
in two important respects:
\begin{itemize}
\item{\emph{Sharpness:}}
Unlike the KKL theorem,
they are not sharp up to constant factors.
\item{\emph{Applicability:}}
They are only effective in the
`dense setting' when $\mu_p(f)$
is bounded away from $0$ and $1$,
whereas the `sparse setting' $\mu_p(f)=o(1)$
is needed for many important open problems.
\end{itemize}
In particular, a sparse analogue of the KKL theorem is a key missing
ingredient in a strategy of Kahn and Kalai \cite{kahn2007thresholds}
for settling their well-known conjecture relating
critical probabilities to expectation thresholds.

\subsection*{Main result}

The fundamental new contribution of this paper
is a hypercontractive theorem for functions that are `global'
(in a sense made precise below). This has many applications,
of which the most significant are as follows.
\begin{itemize}
\item We strengthen Bourgain's Theorem
by obtaining an analogue of the KKL theorem
that is both quantitively tight and
applicable in the sparse regime.
\item We obtain a sharp threshold result
for global monotone functions in the spirit
of the Kahn-Kalai conjecture, bounding the ratio between
the critical probability (where $\mu_p(f)=\frac{1}{2}$)
and the smallest $p$ for which $\mu_p(f)$ is non-negligible.
\item We obtain a $p$-biased generalisation of
the seminal invariance principle
of Mossel, O'Donnell and Oleszkiewicz \cite{mossel2010noise}
(itself a generalisation of the Berry-Esseen theorem
from linear functions to polynomials of bounded degree),
thus opening the door to $p$-biased versions of its
many striking applications in Hardness of Approximation
and Social Choice Theory
(see O'Donnell \cite[Section 11.5]{o2014analysis})
and Extremal Combinatorics (see Dinur--Friedgut--Regev
\cite{dinur2008independent}).
\end{itemize}

\subsection{Hypercontractivity of global functions} \label{subsec:hyper}

Before formally stating our main theorem,
we start by recalling (the $p$-biased version of)
the classical hypercontractive inequality.
Let\footnote{The case where $p>\frac{1}{2}$ is similar.}
$p \in \left(0,\frac{1}{2}\right]$.
For $r \ge 1$ we write $\|\cdot\|_r$
(suppressing $p$ from our notation)
for the norm on $L^r(\{0,1\}^n,\mu_p)$.

\begin{defn}[Noise operator]
  For  $x \in \{0,1\}^n$
  we define the $\rho$-correlated distribution
$N_{\rho}(x)$ on $\left\{ 0,1\right\} ^{n}$: a sample
$\mathbf{y}\sim N_\rho(x)$ is obtained by, independently for each $i$
setting $\mathbf{y}_i=x_i$  with probability $\rho$,
or otherwise (with probability $1-\rho$) we resample $\mathbf{y}_i$
with $\mb{P}(\mathbf{y}_i=1)=p$. We define
the noise operator $\mathrm{T}_\rho$ on $L^2(\{0,1\}^n,\mu_p)$ by
\[
\mathrm{T}_{\rho}\left(f\right)\left(x\right)=\mathbb{E}_{\yb
\sim N_{\rho}\left(x\right)}\left[f\left(\yb\right)\right].
  \]
\end{defn}

H\"older's inequality gives $\|f\|_r\le\|f\|_s$ whenever $r\le s$.
The hypercontractivity theorem gives an inequality
in the other direction after applying noise to $f$;
for example, for $p=1/2$, $r=2$ and $s=4$ we have
\[\|\mathrm{T}_{\rho}f\|_4 \le \|f\|_2\]
for any $\rho\le\frac{1}{\sqrt{3}}$.
A similar inequality also holds when $p=o(1)$,
but the correlation $\rho$ has to be so small
that it is not useful in applications;
e.g.\  if $f(x)=x_1$
(the `dictator' or `half cube'),
then $\|f\|_2 = \sqrt{\mu_p(f)} = \sqrt{p}$ and
$\mathrm{T}_{\rho}f(x)
= \mathbb{E}_{\yb \sim N_{\rho}\left(x\right)} \mathbf{y}_1
= \rho x_1 + (1-\rho)p$,
so $\|\mathrm{T}_{\rho}f\|_4
> (\mb{E}[\rho^4 x_1^4])^{1/4} = \rho p^{1/4}$.
Thus we need $\rho = O(p^{1/4})$ to obtain any
hypercontractive inequality for general $f$.

\subsection*{Local and global functions}
To resolve this issue, we note that the tight examples
for the hypercontractive inequality are \emph{local},
in the sense that a small number of
coordinates can significantly influence the output of the function.
On the other hand, many functions of interest are \emph{global},
in the sense that a small
number of coordinates can change the output of the function only with a
negligible probability; such global functions appear naturally
in Random Graph Theory \cite{achlioptas1999sharp},
Theoretical Computer Science \cite{friedgut1999sharp}
and Number Theory \cite{friedgut2016sharp}.
Our hypercontractive inequality will show that
constant noise suffices for functions that are global
in a sense captured by \emph{generalised influences},
which we will now define.

Let $f\colon\left\{ 0,1\right\} ^{n}\to\mathbb{R}$.
For $S\sub[n]$ and $x\in\{0,1\}^S$, we write $f_{S\to x}$
for the function obtained from $f$ by restricting the
coordinates of $S$ according to $x$
(if $S=\{i\}$ is a singleton we simplify notation to $f_{i\to x}$).
We write $|x|$ for the number of ones in $x$.
For $i \in [n]$, the $i$th influence is
$\mathrm{I}_i(f) = \| f_{i \to 1} - f_{i \to 0} \|_2^2$,
where the norm is with respect to the implicit measure $\mu_p$. In general,
we define the influence with respect to any $S \sub [n]$
by sequentially applying the operators $f\mapsto f_{i\to 1}-f_{i \to
  0}$ for all $i \in S$, as follows.
\begin{defn}
\label{def:generalised influences}
For $f\colon\left\{ 0,1\right\} ^{n}\to\mathbb{R}$ and $S \sub [n]$
we let (suppressing $p$ in the notation)
\[
\mathrm{I}_{S}\left(f\right)=\mathbb{E}_{\mu_{p}}\bigg [\Big (\sum_{x\in\left\{ 0,1\right\} ^{S}}\left(-1\right)^{\left|S\right|-\left|x\right|}f_{S\to x}\Big )^{2}\bigg ].
\]
We say $f$ has \emph{$\bB$-small generalised influences} if
$\mathrm{I}_S(f)\le\bB\ \mb{E}[f^2]$ for all $S\subseteq\left[n\right].$
\end{defn}

The reader familiar with the KKL theorem and the invariance principle
may wonder why it is necessary to introduce generalised influences
rather than only considering influences (of singletons).
The reason is that under the uniform measure
the properties of having small influences
or small generalised influences are qualitatively equivalent,
but this is no longer true in the $p$-biased setting
for small $p$ (consider $f(x) =
\frac{x_1x_2 +\cdots +x_{n-1}x_n}{\|x_1x_2 +\cdots +x_{n-1}x_n\|}$).

We are now ready to state our main theorem,
which shows that global\footnote{Strictly speaking,
our assumption is stronger than the most natural notion
of global functions: we require all generalised influences
to be small, whereas a function should be considered global
if it has small generalised influences
$I_S(f)$ for small sets $S$.
However, in practice, the hypercontractivity Theorem
is typically applied to low-degree truncations of Boolean functions
(see Section \ref{sec:practice}) ,
when there is no difference between these notions,
as $I_S(f)=0$ for large $S$.} functions are
hypercontractive for a noise operator with a constant rate.
Moreover, our result applies to general $L^r$ norms and product spaces
(see Section \ref{sec:hyp}), but for simplicity here we just highlight
the case of $(4,2)$-hypercontractivity in the cube.

\begin{thm} \label{thm:Hypercontractivity}
Let $p \in \left(0,\frac{1}{2}\right]$.
Suppose $f\in L^{2}\left(\left\{ 0,1\right\} ^{n},\mu_{p}\right)$
has $\bB$-small generalised influences (for $p$).
Then $\|\mathrm{T}_{1/5} f\|_{4} \le \bB^{1/4} \|f\|_2$.
\end{thm}

We now move on to demonstrate the power
of global hypercontractivity in the contexts
of isoperimetry, noise sensitivity, sharp thresholds, and invariance.
We emphasise that Theorem \ref{thm:Hypercontractivity} is the only new
ingredient required for these applications, so we expect that it will
have many further applications to generalising results proved via
usual hypercontractivity on the cube with uniform measure.

\subsection{Isoperimetry and influence} \label{subsec:isoinf}

Stability of isoperimetric problems is a prominent open problem at the interface of Geometry, Analysis and Combinatorics.
This meta-problem is to characterise sets whose boundary is close to
the minimum possible given their volume; there are many specific problems
obtained by giving this a precise meaning. Such results in Geometry
were obtained for the classical setting of
Euclidean Space by Fusco, Maggi and Pratelli \cite{fusco2008iso}
and for Gaussian Space by Mossel and Neeman \cite{mossel2015robust}.

The relevant setting for our paper is that of the cube $\{0,1\}^n$,
endowed with the $p$-biased measure $\mu_p$. We refer to this problem
as the ($p$-biased) edge-isoperimetric stability problem.
We identify any subset of $\{0,1\}^n$ with its characteristic
Boolean function $f:\{0,1\}^n \to \{0,1\}$,
and define its `boundary' as the (total) influence\footnote{
Everything depends on $p$, which we fix and suppress in our notation.}
\[\mathrm{I}\left[f\right] =\sum_{i=1}^{n}\mathrm{I}_{i}\left[f\right],
\text{ where each }
\mathrm{I}_i\left[f\right]=\Pr_{\xb\sim\mu_{p}}
\left[f\left(\xb\oplus e_{i}\right)\ne f\left(\xb\right)\right],\]
i.e.\ the \emph{$i$th influence} $\mathrm{I}_i\left[f\right]$
of $f$ is the probability that $f$ depends on bit $i$
at a random input according to $\mu_p$.
(The notion of influence for real-valued functions,
given in Section \ref{subsec:hyper}, coincides with this notion
for Boolean-valued functions).
When $p=1/2$ the total influence corresponds to the classical
combinatorial notion of edge-boundary\footnote{
For the vertex boundary, stability results
showing that approximately isoperimetric sets
are close to Hamming balls were obtained independently by
Keevash and Long \cite{keevash2018vertexiso}
and by Przykucki and Roberts \cite{prz2018vertexiso}.}.

The KKL theorem of Kahn, Kalai and Linial \cite{kahn1988influence}
concerns the structure of functions $f:\{0,1\}^n \to \{0,1\}$,
considering the cube under the uniform measure,
with variance bounded away from $0$ and $1$
and with total influence is upper bounded by some number $K$.
It states that $f$ has a coordinate with influence
at least $e^{-O\left(K\right)}$.
The tribes example of Ben-Or and Linial \cite{ben1990collective}
shows that this is sharp.

\subsection*{$p$-biased versions}
The $p$-biased edge-isoperimetric stability problem is
somewhat understood in the \emph{dense regime}
(where $\mu_p\left(f\right)$ is bounded away from $0$ and $1$)
especially for Boolean functions $f$ that are \emph{monotone}
(satisfy $f\left(x\right)\le f\left(y\right)$
whenever all $x_{i}\le y_{i}$).
Roughly speaking, most edge-isoperimetric stability results in the dense regime say that Boolean functions of small influence have some
`local' behaviour (see the seminal works
of Friedgut--Kalai \cite{friedgut1996every},
Friedgut \cite{friedgut1998boolean, friedgut1999sharp},
Bourgain \cite[Appendix]{friedgut1999sharp},
and Hatami \cite{hatami2012structure}).
In particular, Bourgain
(see also \cite[Chapter 10]{o2014analysis})
showed that for any monotone Boolean function $f$
with $\mu_{p}\left(f\right)$ bounded away from $0$ and $1$ and
$pI\left[f\right]\le K$
there is a set $J$ of $O\left(K\right)$ coordinates
such that $\mu_{p}\left(f_{J\to1}\right)
\ge\mu_{p}\left(f\right)+e^{-O\left(K^{2}\right)}$.
This result is often interpreted as `almost isoperimetric (dense) subsets of the
$p$-biased cube must be local' or on the contrapositive as `global
functions have large total influence'. Indeed, if a
restriction of a small set of coordinates can significantly boost the $p$-biased measure of a function, then this intuitively means
that it is of a local nature.

For monotone functions, the conclusion in Bourgain's theorem
is equivalent (see Section \ref{sec:equiv})
to having some set $J$ of size $O(K)$ with
$\mathrm{I}_J\left(f\right) \ge e^{-O\left(K^2\right)}$.
Thus Bourgain's theorem can be viewed as a
$p$-biased analog of the KKL theorem, where influences are replaced
by generalised influences. However, unlike the KKL Theorem,
Bourgain's result is not sharp,
and the anti-tribes example of Ben-Or and Linial only shows that
the $K^2$ term in the exponent cannot drop below $K$.

As a first application of our hypercontractivity theorem we replace the
term $e^{-O(K^2)}$ by the term $e^{-O(K)}$, which is sharp by Ben-Or
and Linial's example, see Section \ref{sec:equiv}.

\begin{thm}
\label{thm:Bourgain+}
Let $p\in\left(0,\frac{1}{2}\right]$, and let  $f\colon\left\{ 0,1\right\} ^{n}\to\left\{ 0,1\right\} $
be a monotone Boolean function
with $\mu_{p}\left(f\right)$ bounded away from $0$ and $1$
and $\mathrm{I}\left[f\right]\le\frac{K}{p}$.
Then there is a set $J$ of $O\left(K\right)$ coordinates such that
$\mu_{p}\left(f_{J\to1}\right)\ge\mu_{p}\left(f\right)+e^{-O\left(K\right)}$.
\end{thm}

For general functions we prove a similar result, where the conclusion
$\mu_{p}\left(f_{J\to1}\right)\ge\mu_{p}\left(f\right)+e^{-O\left(K\right)}$
is replaced with $\mathrm{I}_J\left(f\right)\ge e^{-O\left(K\right)}$.

\subsection*{The sparse regime}

On the other hand, the \emph{sparse regime}
(where we allow any value of $\mu_p(f)$)
seemed out of reach of previous methods in the literature.
Here Russo \cite{russo1982approximate}, and independently Kahn and Kalai \cite{kahn2007thresholds}, gave a proof of the
$p$-biased isoperimetric inequality:
$p\mathrm{I}\left[f\right]\ge
\mu_p\left(f\right)\log_p\left(\mu_p\left(f\right)\right)$ for every
$f$. They also showed that equality holds only for the monotone
sub-cubes. Kahn and Kalai posed the problem of determining
the structure of monotone Boolean functions $f$ that they called
\emph{$d$-optimal}, meaning that $p\mathrm{I}\left[f\right]\le
d\mu_p\left(f\right)\log_p\left(\mu_p\left(f\right)\right)$,
i.e.\ functions with total influence within a certain
multiplicative factor of the minimal value guaranteed
by the isoperimetric inequality.
They conjectured in \cite[Conjecture 4.1(a)]{kahn2007thresholds}
that for any constant $C>0$
there are constants $K,\dD>0$ such that
if $f$ is $C\log\left(1/p\right)$-optimal
then there is a set $J$ of
$\le K \log \frac{1}{\mu_{p}\left(f\right)}$ coordinates such that
$\mu_{p}\left(f_{J\to1}\right)\ge (1+\dD)\mu_p(f)$.

The corresponding result with a similar conclusion was open
even for $C$-optimal functions! Our second theorem is a variant
of the Kahn--Kalai conjecture which applies to
$C\log\left(1/p\right)$-optimal functions
when $C$ is sufficiently small
(whereas the conjecture requires
an arbitrary constant $C$).
We compensate for our stronger hypothesis in the following result
by obtaining
a much stronger conclusion than that asked for by Kahn and Kalai;
for example, if $f$ is $\frac{\log\left(1/p\right)}{100C}$-optimal
then $\mu_{p}\left(f_{J\to1}\right)\ge \mu_p(f)^{0.01}$.
We will also show that our result
is sharp up to the constant factor $C$.

\begin{thm}
\label{thm:Variant of Kahn kalai}
Let $p\in\left(0,\frac{1}{2}\right]$, $K \ge 1$
and let $f$ be a Boolean function with
$p\mathrm{I}\left[f\right] < K\mu_{p}\left(f\right)$.
Then there is a set $J$ of $\le CK$ coordinates,
where $C$ is an absolute constant, such that
$\mu_{p}\left(f_{J\to1}\right)\ge e^{-CK}$.
\end{thm}

\subsection{Sharp thresholds} \label{sec:kahn kalai}

The results of Friedgut and Bourgain mentioned above
also had the striking consequence that any `global'
Boolean function has a sharp threshold, which was
a breakthrough in the understanding of this phenomenon,
as it superceded many results for specific functions.

The sharp threshold phenomenon concerns the behaviour
of $\mu_p(f_n)$ for $p$ around the critical probability,
defined as follows.
Consider any sequence $f_n\colon\left\{0,1\right\}^n \to \{0,1\}$
of monotone Boolean functions. For $t \in [0,1]$
let $p_n(t) = \inf \{ p: \mu_p(f_n) \ge t\}$.
In particular, $p^c_n := p_n(1/2)$ is
commonly known as the `critical probability'
(which we think of as small in this paper).
A classical theorem of
Bollob\'{a}s and Thomason \cite{bollobas1987threshold}
shows that for any $\eps>0$ there is $C>0$
such that $p_n(1-\eps) \le C p_n(\eps)$.
This motivates the following definition:
we say that the sequence $(f_n)$
has a \emph{coarse threshold}
if for each $\eps>0$ the length of the interval
$[p_n(\eps),p_n(1-\eps)]$ is $\Theta(p^c_n)$,
otherwise we say that it has a \emph{sharp threshold}.

The classical approach for understanding sharp thresholds is based
on the Margulis--Russo formula
$\frac{d\mu_{p}\left(f\right)}{dp}=\mathrm{I}_{\mu_{p}}\left(f\right)$,
see \cite{margulis1977} and \cite{russo1982approximate}.
Here we note that if $f$ has a coarse threshold, then by the Mean
Value Theorem there is a constant $\epsilon>0$, some $p$ with $\mu_p(f)\in(\epsilon,1-\epsilon)$
and $p\mathrm{I}_{\mu_{p}}\left(f\right) = \Tt(1)$,
so one can apply various results mentioned in Section \ref{subsec:isoinf}.
Thus Bourgain's Theorem implies that
there is a set $J$ of $O\left(K\right)$ coordinates
such that $\mu_{p'}\left(f_{J\to1}\right)
\ge\mu_{p'}\left(f\right)+e^{-O\left(K^{2}\right)}$.
While this approach is useful for studying the behaviour of $f$
around the critical probability, it rarely gives any information
regarding the location of the critical probability.
Indeed, many significant papers are devoted to locating
the critical probability of specific interesting functions,
see e.g.\ the breakthroughs of Johansson, Kahn and Vu
\cite{johansson2008factors} and Montgomery
\cite{montgomery2018trees}.

A general result was conjectured
by Kahn and Kalai for the class of Boolean functions of the form
$f_n\colon \{0,1\}^{\binom{[n]}{2}}\to \left\{0,1\right\}$,
whose input is a graph $G$ and whose output is 1
if $G$ contains a certain fixed graph $H$.
For such functions there is a natural
`expectation heuristic' $p^E_n$ for the critical probability,
namely the least value of $p$ such that the expected number of copies
of any subgraph of $H$ in $G\left(n,p\right)$ is at least $1/2$.
Markov's inequality implies $p^c_n \ge p^E_n$,
and the hope of the Kahn--Kalai Conjecture is that there
is a corresponding upper bound up to some multiplicative factor.
They conjectured in \cite[Conjecture 2.1]{kahn2007thresholds}
that $p^c_n = O\left(p^E_n \log n\right)$, but this is widely open,
even if $\log n$ is replaced by $n^{o(1)}$.

The proposed strategy of Kahn and Kalai to this conjecture
via isoperimetric stability is as follows.
\begin{itemize}
\item{Prove a lower bound on $\mu_{p^E_n}\left(f_n\right)$.}
\item{Show (e.g.\ via Russo's lemma)
that if $\left|\left[p_E,p_c\right]\right|$ is too large,
then the $p$-biased total influence at some point
in the interval $\left[p_E,p_c\right]$ must be relatively small.}
\item{Prove an edge-isoperimetric stability result that
      rules out the latter possibility.}
 \end{itemize}

Theorem \ref{thm:Variant of Kahn kalai}
makes progress on the third ingredient.
Combining it with Russo's Lemma, we obtain
the following result that can be used
to bound the critical probability.
Let $f$ be a monotone Boolean function.
We say that $f$ is \emph{$M$-global} in an interval $I$
if for each set $J$ of size $\le M$ and each $p\in I$ we have
$\mu_p\left(f_{J\to1}\right)\le \mu_p\left(f\right)^{0.01}$.

\begin{thm}\label{thm:Sharp threshold result}
  There exists an absolute constant $C$ such that the following
  holds for any monotone Boolean function $f$
  with critical probability $p_c$ and $p\le p_c$.
  Suppose for some $M>0$ that $f$ is $M$-global
  in the interval  $\left[p,p_c\right]$
  and that $\mu_{p}\left(f\right)\ge  e^{-M/C}$.
  Then $p_c\le M^C p$.
\end{thm}

To see the utility of Theorem \ref{thm:Sharp threshold result},
imagine that one wants to bound the critical probability as
$p^c_n \le p$, but instead of showing $\mu_p(f_n)\ge \frac{1}{2}$
one can only obtain a weaker lower bound
$\mu_p\left(f\right)\ge e^{-M/C}$, where $f$ is $M$-global;
then one can still bound the critical probability
as $p^c_n \le M^{O(1)} p$.

\subsection{Noise sensitivity} \label{subsec:noise}

Studying the effect of `noise' on a Boolean function
is a fundamental paradigm in various contexts,
including hypercontractivity (as in Section \ref{subsec:hyper})
and Gaussian isoperimetry
(via the invariance principle, see Section \ref{sec:inv}).
Roughly speaking, a function $f$ is `noise sensitive'
if $f(x)$ and $f(y)$ are approximately independent
for a random input $x$ and random small perturbation $y$ of $x$;
an equivalent formulation (which we adopt below)
is that the `noise stability'
of $f$ is small (compared to $\mu_p\left(f\right)$).
Formally, we use the following definition.

\begin{defn}

The noise stability $\mathrm{Stab}_{\rho}(f)$
of $f \in L^2(\{0,1\}^n,\mu_p)$ is defined by
\[
\mathrm{Stab}_{\rho}\left(f\right)=\left\langle f,\mathrm{T}_{\rho}f\right\rangle =\mathbb{E}_{\xb\sim\mu_{p}}\left[f\left(\xb\right)\mathrm{T}_{\rho}f\left(\xb\right)\right].
\]
A sequence $f_n$ of Boolean functions is said to be noise sensitive if
for each fixed $\rho$ we have $\mathrm{Stab}_{\rho}\left(f_n\right)=\mu_p\left(f_n\right)^2+o\left(\mu_p\left(f_n\right)\right).$
\end{defn}

Note that everything depends on $p$, but this will be clear from the context,
so we suppress $p$ from the notation $\mathrm{Stab}_{\rho}$.
Kahn, Kalai, and Linial \cite{kahn1988influence}
(see also \cite[Section 9]{o2014analysis})
showed that sparse subsets of the uniform cube
are noise sensitive, where we recall that the sequence
$(f_n)$ is \emph{sparse} if $\mu_{p}\left(f_n\right)=o\left(1\right)$ and
\emph{dense} if $\mu_{p}\left(f_n\right)=\Theta\left(1\right)$.

The relationship between noise and influence
in the cube under the uniform measure was further studied
by Benjamini, Kalai, and Schramm \cite{benjamini1999percolation}  (with applications to percolation), who gave
a complete characterisation:
a sequence $(f_n)$ of monotone dense Boolean functions
is noise sensitive if and only if the
sum of the squares of the influences of $f_n$ is $o\left(1\right)$.
Schramm and Steif \cite{schramm2010quantitative} proved that
any dense Boolean function on $n$ variables that can be computed by an
algorithm that reads $o\left(n\right)$ of the input bits is noise
sensitive. Their result had the striking application that the set of
exceptional times in dynamical critical site percolation on the
triangular lattice, in which an infinite cluster exists, is of Hausdorff
dimension in the interval  $\left[\frac{1}{6},\frac{31}{36}\right]$.
Ever since, noise sensitivity was considered in many other contexts
(see e.g.\ the recent results and open problems of Lubetzky--Steif \cite{Lubetzky2015strong} and Benjamini \cite{benjamini2019noise}).

\subsection*{The $p$-biased setting}

In contrast to the uniform setting, in the $p$-biased setting
for small $p$ it is no longer true that sparse sets
are noise sensitive (e.g.\ consider dictators).
Our main contribution to the theory of noise sensitivity is showing
that `global' sparse sets are noise sensitive. Formally,
we say that a sequence $f_n$ of sparse Boolean functions
is \emph{weakly global} if for any $\eps, C > 0$ there is $n_0>0$
so that $\mu_p\left(\left(f_n\right)_{J\to 1}\right) < \eps$
for all $n>n_0$ and $J$ of size at most $C$.

\begin{thm}
  \label{thm:Noise sensitivity}
  Any weakly global sequence of Boolean functions is noise sensitive.
\end{thm}

\subsection{Further applications of global hypercontractivity}

Besides the applications of Theorem \ref{thm:Hypercontractivity}
to isoperimetry, sharp thresholds and noise sensitivity
discussed above, in Section \ref{sec:inv} we will also
generalise the Invariance Principle of
Mossel, O'Donnell and Oleszkiewicz \cite{mossel2010noise}
to the $p$-biased setting:
we show that if a low degree function on the $p$-biased cube
is global (has small {generalised} influences)
then it is close in distribution to a low
degree function on Gaussian space.
There are many other applications
that we defer to future papers:
\begin{itemize}
\item \emph{Exotic settings:}
    Noise sensitivity of sparse sets is related to small-set expansion
    on graphs, which has found many applications in Computer Science.
    Here the interpretation of Theorem \ref{thm:Noise sensitivity}
    is that although not all small sets in the $p$-biased cube expand,
    global small sets do expand.
    Results of a similar nature were proved for the Grassman graph
    (see \cite{subhash2018pseudorandom}) and the Johnson graph
    (see \cite{khot2018johnson}). The former result
    was essential in the proof of the $2$-to-$2$ Games
    Conjecture, a prominent problem in the field of hardness of
    approximation. Both these works involve long calculations, and have
    sub-optimal parameters. In subsequent works
    \cite{ellis2019hypercontractivity,filmus2019hypercontractivity_slice,
      filmus2019hypercontractivity_symmetric, keevash2019random}
     hypercontractive  results for global functions
     are proven for various domains by reducing to
    the $p$-biased cube and using
    Theorem \ref{thm:Hypercontractivity}. The results of
    \cite{ellis2019hypercontractivity,filmus2019hypercontractivity_slice}
    imply the
    corresponding results about small expanding sets in the
    Grassman/Johnson graph with optimal
    parameters. A similar result was also established for a certain
    noise operator on the symmetric group
    \cite{filmus2019hypercontractivity_symmetric}.

\item \emph{Extremal Combinatorics:}
 The junta method, introduced by Dinur and Friedgut \cite{DinurFriedgut}
 and further developed by Keller and Lifshitz \cite{keller2017junta},
 is a powerful tool for solving problems in Extremal Combinatorics
 via the sharp threshold phenomenon. Specifically, it is
 useful for the study of the Tur\'{a}n problem for hypergraphs,
 where one asks how large can a $k$-uniform hypergraph on $n$
 vertices be if it does not contain a copy of a given hypergraph $H$.
 This method was applied in \cite{keller2017junta} to resolve many
 such questions for a wide class of hypergraphs
 called \emph{expanded hypergraphs} in which the edge uniformity
 can be linear in $n$, although the number of edges in $H$ is fixed.
 In a companion paper \cite{keevash2019STandEH}, we apply
 the sharp threshold technology developed in the current paper
 to the regime  where the number of edges of $H$ can grow with $n$,
 thus settling many cases of
 the Huang--Loh--Sudakov conjecture \cite{huang2012size}
 on cross matchings in uniform hypergraphs and the
 F\"uredi--Jiang--Seiver conjecture on path expansions.
 \end{itemize}

The organisation of this paper is as follows.
After introducing some background on Fourier analysis on the cube
in the next section, we prove Theorem \ref{thm:Hypercontractivity}
in Section \ref{sec:hyp}. In Section \ref{sec:equiv} we establish
the equivalence between the two notions of globalness referred to above,
namely control of generalised influences
and insensitivitity of the measure
under restriction to a small set of coordinates.
Section \ref{sec:noise+sharp} concerns the total influence
of global functions, and includes the proofs of
our stability results for the isoperimetric inequality
(Theorems \ref{thm:Bourgain+} and \ref{thm:Variant of Kahn kalai})
and our first sharp threshold result
(Theorem \ref{thm:Sharp threshold result}).
In Section \ref{sec:noise} we prove our result
on noise sensitivity and apply this to deduce
an alternative sharp threshold result.
Section \ref{sec:genhyp} generalises our hypercontractivity result
in two directions: we consider general norms and general product spaces.
In Section \ref{sec:inv} we prove our $p$-biased version of the
Invariance Principle and sketch its application to a variant
of the `Majority is Stablest' theorem and a sharp threshold
result for almost monotone functions.
We end with some concluding remarks.

\section{Notations} \label{sec:not}

Here we summarise some notation
and basic properties of Fourier analysis on the cube.
We fix $p \in (0,1)$ and suppress it in much of our notation,
i.e.\ we consider $\{0,1\}^n$ to be equipped with the $p$-biased
measure $\mu_p$, unless otherwise stated. We let $\sigma = \sqrt {p(1-p)}$
(the standard deviation of a $p$-biased bit). For each $i\in [n]$
we define $\chi_{i}\colon\left\{ 0,1\right\} ^{n}\to {\mathbb R}$
by $\chi_{i}\left(x\right)=\frac {x_i-p}{\sigma }$
(so $\chi_i$ has mean $0$ and variance $1$).
We use the orthonormal Fourier basis $\left\{ \chi_{S}\right\}_{S \subset [n]}$
of $L^{2}\left(\left\{ 0,1\right\} ^{n},\mu_{p}\right)$,
where each $\chi_{S}:=\prod_{i\in S}\chi_{i}$.
Any $f: \{0,1\}^n \to {\mathbb R}$ has a unique expression
$f=\sum _{S\sub [n]}\hat{f}(S)\chi _S$
where $\{\hat f(S)\}_{S\subset [n]}$ are the
\emph{$p$-biased Fourier coefficients of }$f$.
Orthonormality gives the Plancherel identity
$\bgen{f,g} = \sum_{S \sub [n]} \hat{f}(S) \hat{g}(S)$.
In particular, we have the Parseval identity
$\mb{E}[f^2] = \|f\|_2^2 = \bgen{f,f}
= \sum_{S \sub [n]} \hat{f}(S)^2$.
For $\mc{F} \sub \{0,1\}^n$ we  define the $\mc{F}$-truncation
$f^{\mc{F}} = \sum_{S \in \mc{F}} \hat{f}(S)\chi _S$.
Our truncations will always be according to some degree threshold $r$,
for which we write $f^{\le r} = \sum_{|S| \le r} \hat{f}(S)\chi _S$.

For $i \in [n]$, the $i$-\emph{derivative} $f_i$
and $i$-\emph{influence} $\mathrm{I}_i(f)$ of $f$ are
\begin{align*}
f_i & =\mathrm{D}_{i}\left[f\right] =\sigma \big (f_{i\to1}-f_{i\to0}\big )
= \sum_{S: i\in S}\hat{f}\left(S\right)\chi_{S\backslash\left\{ i\right\} },
\text{ and} \\
\mathrm{I}_i(f) & = \| f_{i \to 1} - f_{i \to 0} \|_2^2
= \sS^{-2} \mb{E}[f_i^2]
= \tfrac{1}{p(1-p)} \sum_{S: i\in S} \hat{f}(S)^2.
\end{align*}
The \emph{influence} of $f$ is
\begin{align} \label{eq:inf}
\mathrm{I}(f) = \sum_i \mathrm{I}_i(f)
= (p(1-p))^{-1} \sum_S |S| \hat{f}(S)^2.
\end{align}
In general, for $S \sub [n]$, the $S$-\emph{derivative} of $f$ is
obtained from $f$ by sequentially applying $\mathrm{D}_i$ for each $i \in S$, i.e.\
\[
  \mathrm{D}_S(f)
= \sigma ^{|S|} \sum_{x\in\left\{ 0,1\right\} ^{S}}\
 (-1)^{\left|S\right|-\left|x\right|}{f}_{S\to x}
 = \sum _{T: S \subset T} \hat f (T) \chi _{T\setminus S}.
\]
The \emph{$S$-influence} of $f$
(as in  Definition \ref{def:generalised influences}) is
\begin{align}
\label{equation: gen influence Fourier expression}
\mathrm{I}_{S}(f)= \sigma ^{-2|S|}\| \mathrm{D}_S\left(f\right) \|_2^2
= \sigma ^{-2|S|}\sum _{E: S \subset E} \hat f (E)^2.
\end{align}
Recalling that a function $f$ has $\alpha$-small generalised influences
if $\mathrm{I}_{S}(f)\le\alpha \mb{E}[f^2]$ for all $S\subset [n]$,
we see that this is equivalent to
$\mb{E}[\mathrm{D}_S\left(f\right)^2] \le \alpha \sigma^{2|S|} \mb{E}[f^2]$
for all $S \subset [n]$.

  \section{Hypercontractivity of functions with small generalised influences}
\label{sec:hyp}

In this section we prove our hypercontractive inequality
(Theorem \ref{thm:Hypercontractivity}), which is the fundamental result
that underpins all of the results in this paper.

The idea of the proof is to reduce hypercontractivity in $\mu_p$
to hypercontractivity in $\mu_{1/2}$ via the `replacement method'
(the idea of Lindeberg's proof of the Central Limit Theorem, and of
the proof of Mossel, O'Donnell and Oleszkiewicz
\cite{mossel2010noise} of the invariance principle).
Throughout this section we fix $f:\{0,1\}^n \to \mb{R}$
and express $f$ in
the $p$-biased Fourier basis as $\sum_S \hat{f}(S) \chi^p_S$,
where $\chi^p_S = \prod_{i \in S} \chi^p_i$
and $\chi^p_i(x) = \tfrac{x_i-p}{\sS}$
(the same notation as above, except that we introduce the
superscript $p$ to distinguish the $p$-biased and uniform settings).

For $0 \le t \le n$ we define
$f_t = \sum_S \hat{f}(S) \chi_S^t$,  where
\[\chi^t_S = \prod\limits_{i\in S\cap [t]}{\chi^{1/2}_i(x)}
 \prod\limits_{i\in S\setminus [t]}{\chi^{p}_i(x)}
\in L^2(\{0,1\}^{[t]},\mu_{1/2}) \times
 L^2(\{0,1\}^{[n] \sm [t]},\mu_p). \]
Thus $f_t$ interpolates from $f_0=f \in L^2(\{0,1\}^n,\mu_p)$
to $f_n = \sum_S \hat{f}(S) \chi^{1/2}_S
\in L^2(\{0,1\}^n,\mu_{1/2})$.
As $\{ \chi^t_S: S \sub [n]\}$ is an orthonormal basis
we have $\|f_t\|_2 = \|f\|_2$ for all $t$.

We also define noise operators $\mathrm{T}^t_{\rho',\rho}$
on $L^2(\{0,1\}^{[t]},\mu_{1/2}) \times
 L^2(\{0,1\}^{[n] \sm [t]},\mu_p)$
by $\mathrm{T}^t_{\rho',\rho}(g)(\xb) =
\mb{E}_{\yb \sim N_{\rho',\rho}(\xb)}[f(\yb)]$,
where to sample $\yb$ from $N_{\rho',\rho}(\xb)$,
for $i \le t$ we let $y_i=x_i$ with probability $\rho'$
or otherwise we resample $y_i$ from $\mu_{1/2}$,
and for $i>t$ we let $y_i=x_i$ with probability $\rho$
or otherwise we resample $y_i$ from $\mu_p$.
Thus $\mathrm{T}^t_{\rho',\rho}$ interpolates from
$\mathrm{T}^0_{\rho',\rho}=\mathrm{T}_{\rho}$ (for $\mu_p$) to
$\mathrm{T}^n_{\rho',\rho}=\mathrm{T}_{\rho'}$ (for $\mu_{1/2}$).

We record the following estimate for $4$-norms
of $p$-biased characters:
\[ \lambda := \mb{E}[ (\chi^p_i)^4 ]
= \sS^{-4}( p (1-p)^4 + (1-p) p^4 )
= \sS^{-2}( (1-p)^3 + p^3 ) \le \sS^{-2}. \]

The core of our argument by replacement is the
following lemma which controls the evolution of
$\mb{E}[(\mathrm{T}^t_{2\rho,\rho} f_t)^4]
= \norm{\mathrm{T}^t_{2\rho,\rho} f_t}_4^4$
for $0 \le t \le n$.

\begin{lem} \label{lem:Replacement step}
$\mb{E}[(\mathrm{T}^{t-1}_{2\rho,\rho} f_{t-1})^4]
\le \mb{E}[(\mathrm{T}^t_{2\rho,\rho} f_t)^4]
+ 3\lambda\rho^4 \mb{E}[(\mathrm{T}^t_{2\rho,\rho} ((\mathrm{D}_t f)_t) )^4]$.
\end{lem}

\begin{proof}
We write
\begin{align*}
f_t & = \chi^{1/2}_t g + h \ \ \text{ and } \ \
 f_{t-1} =  \chi^p_t g + h, \ \ \text{ where } \\
g & = (\mathrm{D}_t f)_t
= \sum_{S: t \in S} \hat{f}(S) \chi^t_{S \sm \{t\}}
= \sum_{S: t \in S} \hat{f}(S) \chi^{t-1}_{S \sm \{t\}}
= (\mathrm{D}_t f)_{t-1}, \ \ \text{ and } \\
h & = \mb{E}_{x_t \sim \mu_{1/2}} f_t
= \sum_{S: t \notin S} \hat{f}(S) \chi^t_S
= \sum_{S: t \notin S} \hat{f}(S) \chi^{t-1}_S
= \mb{E}_{x_t \sim \mu_p} f_{t-1}.
\end{align*}
We also write
\begin{align*}
\mathrm{T}^t_{2\rho,\rho} f_t & = 2\rho \chi^{1/2}_t d + e \ \ \text{ and } \ \
\mathrm{T}^{t-1}_{2\rho,\rho} f_{t-1} = \rho \chi^p_t d + e, \ \ \text{ where } \\
d & = \mathrm{T}^t_{2\rho,\rho} g = \mathrm{T}^{t-1}_{2\rho,\rho} g
\ \ \text{ and } \ \
e = \mathrm{T}^t_{2\rho,\rho} h = \mathrm{T}^{t-1}_{2\rho,\rho} h.
\end{align*}
We can calculate the expectations in the statement of the lemma
by conditioning on all coordinates other than $x_t$, i.e.\
$\mb{E}_{\xb}[ \cdot ] = \mb{E}_{\xb'}[ \mb{E}_{x_t}[ \cdot \mid \xb'] ]$
where $\xb'$ is obtained from $\xb=(x_1,\dots,x_n)$ by removing $x_t$.
It therefore suffices to establish the required inequality
for each fixed $\xb'$ with expectations over the choice of $x_t$;
thus we can treat $d$ and $e$ as constants, and it suffices to show
\begin{equation} \label{eq:ets}
\mb{E}_{x_t}[ (\rho d\chi^p_t + e)^4 ]
\le \mb{E}_{x_t}[ (2\rho d\chi^{1/2}_t + e)^4 ]
+ 3\lambda\rho^4 d^4.
\end{equation}
As $\chi^p_t$ has mean $0$, we can expand the left hand side
of \eqref{eq:ets} as
\[ (\rho d)^4 \mb{E}[(\chi^p_t)^4]
 + 4 e (\rho d)^3  \mb{E}[(\chi^p_t)^3]
 + 6 e^2 (\rho d)^2  \mb{E}[(\chi^p_t)^2] + e^4
 \le 3 \lL (d\rho)^4 + 8 (de\rho)^2 + e^4, \]
where we bound the second term
using Cauchy-Schwarz then AM-GM by
\[ 4 \cdot \mb{E}[(d\rho\chi^p_t)^4]^{1/2}
\cdot \mb{E}[(de\rho\chi^p_t)^2]^{1/2}
\le 2 \left( \mb{E}[(d\rho\chi^p_t)^4]
+  \mb{E}[(de\rho\chi^p_t)^2] \right)
= 2( \lL (d\rho)^4 + (de\rho)^2 ). \]
Similarly, as $\mb{E}[\chi^{1/2}_t]
= \mb{E}[(\chi^{1/2}_t)^3] = 0$,
we can expand the first term
on the right hand side of \eqref{eq:ets} as
\[ (2\rho d)^4 \mb{E}[(\chi^{1/2}_t)^4]
 + 6 e^2 (2\rho d)^2  \mb{E}[(\chi^{1/2}_t)^2] + e^4
= (2\rho d)^4 + 6(2\rho de)^2 + e^3
\ge 8 (de\rho)^2 + e^4. \]
The lemma follows.
\end{proof}

Now we apply the previous lemma inductively to prove the following estimate.

\begin{lem} \label{hypinduct}
$\| \mathrm{T}^i_{2\rho,\rho} f_i \|_4^4  \le \sum_{S \sub [n] \sm [i]}
 (3\lL\rho^4)^{|S|} \| \mathrm{T}^n_{2\rho,\rho} ((\mathrm{D}_S f)_n) \|_4^4$
 for all $0 \le i \le n$.
\end{lem}

\begin{proof}
We prove the inequality by induction on $n-i$
simultaneously for all functions $f$.
If $n=i$ then equality holds trivially.
Now suppose that $i<n$.
By Lemma \ref{lem:Replacement step} with $t=i+1$,
and the induction hypothesis applied to $f$ and $\mathrm{D}_t f$
with $i$ replaced by $t$, we have
\begin{align*}
\| \mathrm{T}^i_{2\rho,\rho} f_i \|_4^4
& \le  \| \mathrm{T}^t_{2\rho,\rho} f_t \|_4^4
+ 3\lL\rho^4 \| \mathrm{T}^t_{2\rho,\rho} ((\mathrm{D}_t f)_t) \|_4^4 \\
& \le \sum_{S \sub [n] \sm [t]}
 (3\lL\rho^4)^{|S|} \| \mathrm{T}^n_{2\rho,\rho} ((\mathrm{D}_S f)_n) \|_4^4
+ 3\lL\rho^4 \sum_{S \sub [n] \sm [t]}
 (3\lL\rho^4)^{|S|} \| \mathrm{T}^n_{2\rho,\rho} ((\mathrm{D}_S \mathrm{D}_t f)_n) \|_4^4 \\
& = \sum_{S \sub [n] \sm [i]}
 (3\lL\rho^4)^{|S|} \| \mathrm{T}^n_{2\rho,\rho} ((\mathrm{D}_S f)_n) \|_4^4. & \qedhere
\end{align*}
\end{proof}

In particular, recalling that
$\mathrm{T}^0_{2\rho,\rho}=\mathrm{T}_{\rho}$ on $\mu_p$ and
$\mathrm{T}^n_{2\rho,\rho}=\mathrm{T}_{2\rho}$ on $\mu_{1/2}$,
the case $i=0$ of Lemma \ref{hypinduct} is as follows.

\begin{prop}
\label{prop:Reason for hypercontractivity}
$\| \mathrm{T}_{\rho} f \|_4^4  \le \sum_{S \sub [n]}
 (3\lL\rho^4)^{|S|} \| \mathrm{T}_{2\rho} ((\mathrm{D}_S f)_n) \|_4^4$.
\end{prop}

The $4$-norms on the right hand side of
Proposition \ref{prop:Reason for hypercontractivity}
are with respect to the uniform measure $\mu_{1/2}$,
where we can apply standard hypercontractivity
(the `Beckner-Bonami Lemma') for $\rho \le 1/2\sqrt{3}$
to obtain $\| \mathrm{T}_{2\rho} ((\mathrm{D}_S f)_n) \|_4^4
\le \| (\mathrm{D}_S f)_n \|_2^4 = \| \mathrm{D}_S f \|_2^4 = \sS^{4|S|} \mathrm{I}_S(f)^2$.
Recalling that $\lL \le \sS^{-2}$,
we deduce the following bound for $\| \mathrm{T}_{\rho} f \|_4^4$
in terms of the generalised influences of $f$.

\begin{thm} \label{thm:hypref}
If $\rho\le 1/\sqrt{12}$ then
$\norm{\mathrm{T}_{\rho} f}_4^4  \le \sum_{S \sub [n]}
(3\lL\rho^4)^{|S|} \| \mathrm{D}_S f \|_2^4
\le \sum_{S \sub [n]} (3\sS^2\rho^4)^{|S|} \mathrm{I}_S(f)^2$.
\end{thm}

Now we deduce our hypercontractivity inequality.
It is convenient to prove the following
slightly stronger statement, which implies
Theorem \ref{thm:Hypercontractivity} using
$\|\mathrm{D}_S f\|_2^2 = \sS^{2|S|} \mathrm{I}_S(f)
\le \lL^{-|S|} \mathrm{I}_S(f)$ and
$\|\mathrm{T}_{1/5} f\|_4 \le
\|\mathrm{T}_{1/\sqrt{24}}f\|_{4}$
(any $\mathrm{T}_\rho$ is a contraction
in $L^p$ for any $p \ge 1$).

\begin{thm} \label{thm:hyp+}
Let $f\in L^{2}\left(\left\{ 0,1\right\} ^{n},\mu_{p}\right)$
with all $\|\mathrm{D}_S f \|_{2}^2 \le \bB \lL^{-|S|} \mb{E}[f^2]$. Then
$\|\mathrm{T}_{1/\sqrt{24}}f\|_{4} \le \bB^{1/4} \|f\|_{2}$.
\end{thm}

\begin{proof}
By Theorem \ref{thm:hypref} applied to $\mathrm{T}_{1/\sqrt{2}} f$
with $\rho=1/\sqrt{12}$ we have
\[\norm{\mathrm{T}_{1/\sqrt{24}} f}_4^4  \le \sum_{S \sub [n]}
(3\lL\rho^4)^{|S|} \| \mathrm{D}_S \mathrm{T}_{1/\sqrt{2}} f \|_2^4.\]
As $\| \mathrm{D}_S \mathrm{T}_{1/\sqrt{2}} f \|_2^2
= \sum_{E: S \sub E} 2^{-|E|} \hat{f}(E)^2
\le \sum_{E: S \sub E} \hat{f}(E)^2
= \| \mathrm{D}_S f \|_2^2 \le \bB \lL^{-|S|} \mb{E}[f^2]$
we deduce
\begin{align*}
\norm{\mathrm{T}_{1/\sqrt{24}} f}_4^4  \le \sum_{S \sub [n]}
\sum_{E: S \sub E} \bB \mb{E}[f^2] 2^{-|E|} \hat{f}(E)^2
= \bB \mb{E}[f^2] \sum_E \hat{f}(E)^2 = \bB \|f\|_2^4.
& \qedhere
\end{align*}
\end{proof}

\subsection{Hypercontractivity in practice} \label{sec:practice}
We will mostly use the following application of the hypercontractivity
theorem.
\begin{lem} \label{lem:applying_hypercontractivity}
  Let $f$ be a function of degree $r$. Suppose that
  $\mathrm{I}_S(f)\le \delta$ for all $|S|\le r$. Then \[
    \|f\|_4\le
    5^{\frac{3r}{4}}\delta^{\frac{1}{4}}\left\| f \right\|_2^{0.5}.
    \]
\end{lem}

The proof uses the following lemma, which is immediate from the Fourier expression in \eqref{equation: gen influence Fourier expression}.

\begin{lem} \label{obs}
$\mathrm{I}_S(f^{\le r}) \le \mathrm{I}_S(f)$
for all $S \sub \left[n\right]$ and $\mathrm{I}_S(f^{\le r})=0$ if $\left|S\right|>r$.
\end{lem}

\begin{proof}[Proof of Lemma \ref{lem:applying_hypercontractivity}]
Write $f = \mathrm{T}_{1/5}(h)$, where
$h = \sum_{|T| \le r} 5^{|T|} \hat{f}(T) \chi_T$.
We will bound the 4-norm of $f$ by applying Theorem
\ref{thm:Hypercontractivity} to $h$,
so we need to bound the generalised influences of $h$.

By Lemma \ref{obs}, for $S \sub [n]$ we have
$\mathrm{I}_S(h)=0$ if $|S|>r$. For $|S|\le r$,  we have
\[
 \mathrm{I}_S(h) = \sS^{-2|S|} \sum_{T: S \sub T, |T| \le r}
 5^{2|T|} \hat{f}(T)^2
\le 5^{2r} \mathrm{I}_S(f) \le 5^{2r} \delta
= \alpha \|h\|_2^2,
\]
where
$\alpha = 5^{2r} \delta /\|h\|_2^2$.
By Theorem \ref{thm:Hypercontractivity}, we have
\[
  \|f \|_4 = \|\mathrm{T}_{1/5} h \|_{4}
\le \alpha^{\frac{1}{4}} \| h \|_{2}
= 5^{r/2} \delta^{\frac{1}{4}}  \sqrt{ \|h\|_2}
\le 5^{\frac{3r}{4}} \delta^{\frac{1}{4}}\sqrt{\|f\|_2}.
\]
In the final inequality we used $\|h\|_2\leq 5^r \|f\|_2$,
which follows from Parseval.
\end{proof}

\section{Equivalence between globalness notions} \label{sec:equiv}

Above we have introduced two notions of what it means for
a Boolean function $f$ to be global.
The first globalness condition,
which appears e.g.\ in Theorem \ref{thm:Bourgain+},
is that the measure of $f$ is not sensitive
to restrictions to small sets of coordinates.
The second condition is a bound on generalised
influences $I_S(f)$ for small sets $S$.
In this section we show that we can move
freely between these notions for two classes
of Boolean functions:
namely sparse ones and monotone ones.

Throughout we assume $p \le 1/2$, which does not
involve any loss in generality in our main results;
indeed, if $p>1/2$ we can consider
the dual $f^*(x) = 1-f(1-x)$ of any Boolean function $f$,
for which $\mu_{1-p}(f^*) = 1-\mu_p(f)$
and $\mathrm{I}_{\mu_{1-p}}(f^*) = \mathrm{I}_{\mu_p}(f)$.

We start by formalising our first notion of globalness.

\begin{defn} \label{def:global}
 We say that a Boolean function $f$
 is \emph{$\left(r,\delta\right)$-global} if
 $\mu_p\left(f_{J\to 1}\right)\le \mu_p\left(f\right)+\delta$
 for each set $J$ of size at most $r$.
\end{defn}

We remark that Definition \ref{def:global}
is a rather weak notion of globalness,
so it is quite surprising that it suffices
for Theorems \ref{thm:Variant of Kahn kalai}
and \ref{thm:Noise sensitivity},
where one might have expected to need
the stricter notion that
$\mu_p(f_{J\to 1})$ is close to $\mu_p(f)$.

The following lemma shows that if a sparse Boolean function
is global in the sense of Definition \ref{def:global}
then it has small generalised influences.

\begin{lem}
\label{lem: gen influence control for global functions}
Suppose that $f: \{0,1\}^n \to \{0,1\}$
is an  $\left(r,\delta\right)$-global Boolean function
with $\mu_p(f) \le \delta$.
Then $\mathrm{I}_{S}\left(f^{\le r}\right) \le
\mathrm{I}_{S}\left(f\right) \le 8^{r} \delta$
for all $S\sub [n]$ with $|S| \le r$.
\end{lem}

\begin{proof}
The first inequality is from Lemma \ref{obs}. Next, we estimate
\begin{align} \label{eq:s}
\sqrt{\mathrm{I}_{S}\left(f\right)}  =
\left \| \sum_{x \in \{0,1\}^S}\left(-1\right)^{\left|S\right | -
  |x|}f_{S\to x} \right \|_2 \le \sum_{x\in \{0,1\}^S}\left\|f_{S\to
  x}\right \|
  _2=\sum_{x\in\{0,1\}^S}\sqrt{\mu_p(f_{S\to x})}.
\end{align}
Next we fix $x \in \{0,1\}^S$ and claim that
$\mu_p(f_{S\to x}) \le 2^r\delta$.
By substituting this bound in \eqref{eq:s}
we see that this suffices to complete the proof.
Let $T$ be the set of all $i\in S$ such that $x_i=1$.
Since $f$ is nonnegative, we have
$\mu_p(f_{T\to 1})\ge \left(1-p\right)^{\left|S\backslash
    T\right|}\mu_p(f_{S\to x})$.
As $f$ is $\left(r,\delta\right)$-global
and $\mu_p(f) \le \delta$, we have
$\mu_p\left(f_{T\to1}\right)\le 2\delta$,
so $\mu_p(f_{S\to x})\le (1-p) ^{|T|-r}2\delta \le 2^r\delta$,
where for the last inequality we can assume $T \ne \es$,
as $\mu_p\left(f_{T\to 1}\right) = \mu_p(f) \leq\delta \le 2^r\delta$.
This completes the proof.
\end{proof}

Next we show an analogue of the previous lemma
replacing the assumption that $f$ is sparse
by the assumption that $f$ is monotone.

\begin{lem} \label{rem}
Let $f\colon\{0,1\}^n\to\{0,1\}$
be a monotone Boolean $\left(r,\delta\right)$-global function.
Then $\mathrm{I}_S\left(f\right) \le 8^{r}\delta$
for every nonempty $S$ of size at most $r$.
\end{lem}

The proof is based on the following lemma
showing that globalness is inherited
(with weaker parameters) under restriction of a coordinate.

\begin{lem} \label{restrict}
    Suppose that $f$ is a monotone
    $\left(r,\delta\right)$-global function. Then for each  $i$:
    \begin{enumerate}
    \item{$f_{i\to 1}$ is $\left(r-1,\delta\right)$-global,}
    \item{$\mu_p\left(f_{i\to 0}\right) \ge
        \mu_p\left(f\right)-\frac{p\delta}{1-p}$,}
    \item{$f_{i\to 0}$ is $\left( r-1, \frac{\delta}{1-p}\right)$-global}.
     \end{enumerate}
   \end{lem}

    \begin{proof}
 To see (1), note that for any $J$ with $|J| \le r-1$ we have
 $\mu_p((f_{i \to 1})_{J \to 1}) = \mu_p(f_{J \cup \{i\} \to 1})
 \le \mu_p(f) + \delta \le \mu_p(f_{i \to 1}) + \delta$,
 where the last inequality holds as $f$ is monotone.
 Statement (2) follows from the upper bound $\mu_p\left(f_{i\to
        1}\right)\le \mu_p\left(f\right)+\delta$
      and    $\mu_p\left(f_{i\to
          0}\right)=\frac{\mu_p\left(f\right)-p\mu_p\left(f_{i\to1}\right)}{\left(1-p\right)}.$

   For (3), we note that by monotonicity
   $\mu_p\left(\left(f_{i\to
          0}\right)_{S\to 1}\right)
          \le \mu_p\left(f_{\{i\}\cup S\to
        1}\right).$
        As $f$ is $\left(r,\delta\right)$-global,
    \[ \mu_p\left(f_{S\cup\{i\}\to 1}\right)
    \le \mu_p\left(f\right)+\delta
    \le \mu_p\left(f_{i\to 0}\right)+\delta+\frac{p\delta}{1-p}=\mu_p\left(f_{i\to 0}\right)+\frac{\delta}{1-p}, \]
    using (2). Hence, $f_{i\to 0}$ is
    $\left(r,\frac{\delta}{1-p}\right)$-global.
  \end{proof}

\begin{proof}[Proof of Lemma \ref{rem}]
We argue by induction on $r$. In the case where  $r=1$,
Lemma \ref{restrict} and monotonicity of $f$
imply (using $p \le 1/2$)
 \[\mathrm{I}_{i}\left(f\right)=\mu_p\left(f_{i\to 1}\right) -
   \mu_p\left(f_{i\to 0}\right)\le \delta
   +\frac{p\delta}{1-p}\le 2\delta. \]
Now we bound $\mathrm{I}_{S\cup\left\{i\right\}}\left(f\right)$
for $r>1$ and $S$ of size $r-1$ with $i \notin S$.

Note that $\mathrm{D}_{S\cup\{i\}}\left(f\right)=\mathrm{D}_{S}
\left[\mathrm{D}_i (f)\right]$. By the triangle inequality,
we have
 \[
   \sqrt{\mathrm{I}_{S\cup\left\{i\right\}} \left(f\right)} =
   \sigma^{-r}\|\mathrm{D}_{S\cup\left\{i\right\}}(f)\|_2 =
   \sigma^{1-r}\|\mathrm{D}_{S}(f_{i\to 1})-\mathrm{D}_S(f_{i\to 0})\|_2
   \le
   \sqrt{ \mathrm{I}_S\left(
       f_{i\to1} \right) } +
   \sqrt{ \mathrm{I}_S{\left(f_{i\to
       0} \right)}}.
\]
By the induction hypothesis and Lemma \ref{restrict}
 the right hand side is at most
\[
  \sqrt{8^{r-1}\delta}
  +\sqrt{8^{r-1}2\delta}
  \le \sqrt{8^r\delta}.
\]
Taking squares, we obtain  $\mathrm{I}_{S\cup\left\{i\right\}}\left(f\right)\le
8^r\delta.$
\end{proof}

We conclude this section by showing the converse direction of
the equivalence between our two notions of globalness,
i.e.\ that if the generalised influences of a function $f$
are small then $f$ is global in the sense of its measure
being insensitive to restrictions to small sets.
(We will not use the lemma in the sequel
but include the proof for completeness.)

\begin{lem}
  Let $f\colon\{0,1\}^n\to \{0,1\}$ be a Boolean function and let $r>0$. Suppose
  that $\mathrm{I}_S[f] \le \delta$ for each nonempty set $S$ of at most
  $r$ coordinates. Then $f$ is $\left(r,4^{r}\delta\right)$-global.
\end{lem}
\begin{proof}
To facilitate a proof by induction on $r$
we prove the slightly  stronger statement that
$f$ is $(r,\sum_{i=1}^{r}4^{i-1}\delta)$-global.
Suppose first that $r=1$. Our
goal is to show that if $\mathrm{I}_i[f]<\delta$, then
$\mu_p(f_{i\to1})-\mu_p(f_{i\to 0})<\delta$, and indeed,
\[
  \mu_p(f_{i\to1})-\mu_p(f_{i\to 0})
  \le \Pr[f_{i\to 1}\ne f_{i\to
    0}]= \|f_{i\to 1}-f_{i\to 0}\|_2^2
    =\|\mathrm{D}_i[f]\|_2^2=\mathrm{I}_{i}[f] < \delta.
\]
Now suppose that $r>1$ and that the lemma holds with $r-1$ in place of
$r$. The lemma will follow once we show that for all $i$ and all
nonempty sets
$S$ of size at most $r-1$, we have $\mathrm{I}_{S}[f_{i\to 1}]\le 4\delta$.
Indeed, the induction hypothesis and the $n=1$ case
will imply that for each set $S$ of size at most $r$ and each $i\in S$
we have $\mu_p(f_{S\to 1}) \le \mu_p(f_{i\to 1}) +
\sum_{i=1}^{r-1}4^{i-1}\cdot 4\delta \le \mu_p(f) +
\sum_{i=1}^{r}4^{i-1}\delta$.

We now turn to showing the desired upper bound on the generalised
influences of  $f_{i\to 1}$. Let $S$ be a set of size at
most $r-1$. Recall that $\mathrm{I}_S[f_{i\to 1}]=\|
\mathrm{D}_{S}[f_{i\to 1}] \|_2^2$. We may assume that $i\notin S$ for
otherwise the generalised influence $\mathrm{I}_S[f_{i\to 1}]$ is
0. We make two observations. Firstly, we have
\[
  \mathrm{D}_{S\cup \{i\}}[f]=\mathrm{D}_{S}[f_{i\to 1}]
  -\mathrm{D}_S[f_{i\to 0}].
\]
Secondly, conditioning on the ouput of the coordinate $i$ we have
\[
  \|\mathrm{D}_S[f]\|_2^2=p\| \mathrm{D}_S[f_{i\to
    1}]\|_2^2 + (1-p)\|\mathrm{D}_S[f_{i\to 0}]\|_2^2,
\]
which implies
$\|\mathrm{D}_S[f_{i\to 0}]\|_2 \le \sqrt{2}\|\mathrm{D}_S[f]\|_2$. We
may now apply the triangle inequality on the first
observation and use the second observation to obtain
\[
  \sqrt{\mathrm{I}_S[f]} = \| \mathrm{D}_S[f_{i\to 1}]\|_2 \le
  \|\mathrm{D}_{S\cup \{ i\}}[f]\|_2 + \|\mathrm{D}_{S}[f_{i\to
    0}]\|_2\le \sqrt{\delta} + \sqrt{2} \|\mathrm{D}_S[f]\|_2\le 2\sqrt{\delta}.
\]

Taking squares, we obtain the desired upper bound on the generalised
influences of $f_{i\to 1}$.
\end{proof}

\section{Total influence of global functions}
\label{sec:noise+sharp}

In this section we show that our hypercontractive inequality
(Theorem \ref{thm:Hypercontractivity})
implies our stability results for the isoperimetric inequality, namely
Theorems \ref{thm:Bourgain+} and \ref{thm:Variant of Kahn kalai}. We also deduce our first sharp threshold result,  Theorem
\ref{thm:Sharp threshold result}.

\subsection{The spectrum of sparse global sets}\label{spectrum}

The key step in the proofs of Theorems \ref{thm:Variant of Kahn
  kalai} and \ref{thm:Noise sensitivity} is to show that the Fourier
spectrum of global sparse subsets of the $p$-biased cube
is concentrated on the high
degrees. We recall first a proof that in the uniform cube
(i.e.\ cube with uniform measure), \emph{all} sparse sets
have this behaviour (not just the global ones).
Our proof is based on ideas from Talagrand
\cite{talagrand1994approximate} and Bourgain and Kalai
\cite{BourgainKalai19}.

\begin{thm} \label{thm:warm-up}
Let $f$ be a Boolean function on the uniform cube, and let $r>0$. Then
\[
  \left\|f^{\le r} \right\|_2^2\le 3^r \mu_{1/2}\left(f\right)^{1.5}.
  \]
\end{thm}

  The idea of the proof is to  bound $\left\|f^{\le r}
  \right\|_2^2=\left\langle f^{\le r},f\right\rangle$ via H\"older
  by $\left\|f^{\le r}\right\|_4 \left\|f\right\|_{4/3}$,
  bound the $4$-norm via hypercontractivity
  and express the $4/3$-norm in terms of the measure of $f$
  using the assumption that $f$ is Boolean.
  For future reference, we decompose the argument into two lemmas,
  the first of which applies also to the $p$-biased settting
  and the second of which requires hypercontractivity,
  and so is specific to the uniform setting.
  Theorem \ref{thm:warm-up} follows immediately from
  Lemmas \ref{lem:nt} and \ref{lem:upper bound on 4th moment} below.

In the following lemma we consider $\{-1,0,1\}$-valued functions
so that it can be applied to either a Boolean function
or its discrete derivative.

  \begin{lem}\label{lem:nt}
Let $f\colon\power{n}\to\{0,1,-1\}$, let $\mc{F}$ be a family of
subsets of $\left[n\right]$,
and let $g(x) = f^{\mc{F}} = \sum_{S\in\mc{F}}{\hat{f}(S)\chi_S(x)}$.
Then $\|g\|_2^2 \le \|g\|_4 \|f\|_2^{1.5}$, where the norms
can be taken with respect to an arbitrary $p$-biased measure.
\end{lem}

\begin{proof}
By Plancherel and H{\"o}lder's inequality,
$\mb{E}[g^2] = \inner{f}{g}
\le \norm{f}_{4/3}\norm{g}_4$,
where $\norm{f}_{4/3} = \mb{E}[f^2]^{3/4}=\|f\|_2^{1.5}$
as $f$ is $\{-1,0,1\}$-valued.
\end{proof}

  Applying Lemma \ref{lem:nt} with $g=f^{\le r}$, we obtain a lower bound on
  the 4-norm of $g$. We now upper bound it by appealing to the
  Hypercontractivity Theorem.
  \begin{lem}\label{lem:upper bound on 4th moment}
    Let $g$ be a function of degree $r$ on the uniform cube. Then
    $\left\|g\right\|_4 \le \sqrt{3}^{r}\left\|g\right\|_2$.
  \end{lem}
  \begin{proof}
    Let $h$ be the function, such that $\mathrm{T}_{1/\sqrt{3}}h=g$,
    i.e.\ $h=\sum_{|S|\le r}\sqrt{3}^{|S|}\hat{g}\left(
      S\right)\chi_S$. Then the Hypercontractivity
    Theorem implies that $\|g\|_4\le \|h\|_2$, and by Parseval
    $\|h\|_2\le \sqrt{3}^r\|g\|_2$.
    \end{proof}

We shall now adapt the proof of Theorem \ref{thm:warm-up} to global
functions on the $p$-biased cube. The only part in the above proof
that needs an adjustment is Lemma \ref{lem:upper bound on 4th
  moment}, and in fact
  we have already provided the required adjustment in Section
\ref{sec:hyp} in the form of Lemma \ref{lem:applying_hypercontractivity}.

\begin{thm} \label{lem:normsense0}
Let $r \ge 1$,
and let  $f\colon\power{n}\to\{0,1,-1\}$. Suppose that
$\mathrm{I}_S[f^{\le r}]\le \delta$ for each set $S$ of size at most $r$.
Then $\mb{E}[(f^{\le r})^2]
\le 5^r \delta^{\frac{1}{3}}  \mb{E}\left[f^2\right]$.
\end{thm}

\begin{proof}
Applying Lemma \ref{lem:applying_hypercontractivity} with $g=f^{\le r}$, we
obtain the upper bound $\|g\|_4\le 5^{\frac{3r}{4}}\delta^{\frac{1}{4}}\|g\|_2^{0.5}$. Since
the function $f$ takes values only in the set $\{0,1,-1\}$,
we may apply Lemma \ref{lem:nt}. Combining it with the upper bound on
the 4-norm of $g$, we obtain
\[
\|g\|_2^2
\le \|g\|_4\|f\|_2^{1.5}\le 5^{\frac{3r}{4}}\delta^{\frac{1}{4}}
\|g\|_2^{0.5}\|f\|_2^{1.5}.
\]
Rearranging, and raising everything to the power $\frac{4}{3}$, we obtain
$\|g\|_2^2\le 5^{r}\delta^{\frac{1}{3}}\left\|f\right\|_2^2$.
\end{proof}

Let us say that $f$ is \emph{$\epsilon$-concentrated} above degree $r$ if
$\|f^{\le r}\|_2^2 \le \epsilon \|f\|_2^2$. The significance of
Theorem \ref{lem:normsense0} stems from the fact that it implies
the following result showing that for each $r,\epsilon>0$
there exists a $\delta>0$ such that any
sparse $(r,\delta)$-global function
is $\epsilon$-concentrated above degree $r$.

\begin{cor} \label{lem:normsense}
Let $r \ge 1$.
Suppose that  $f$ is an
$\left(r,\delta\right)$-global Boolean function
with $\mu_p\left( f\right)<\delta$.
Then $\mb{E}[(f^{\le r})^2]
\le 10^{r} \delta^{\frac{1}{3}}\mu_p(f)$.
\end{cor}
\begin{proof}
  By Lemma \ref{lem: gen influence
    control for global functions}, for each $S$ of size $r$ we have
  $ \mathrm{I}_S\left(f^{\le r}\right)\le\mathrm{I}_S\left(f\right)< 8^r \delta$.
  Then Theorem \ref{lem:normsense0} implies
  $\|f^{\le r}\|_2^2\le 10^r\delta^{\frac{1}{3}} \|f\|_2^2$,
  where since $f$ is Boolean
  we have $\|f\|_2^2 = \mu_p(f)$.
\end{proof}

\subsection{Isoperimetric stability}

We are now ready to prove our variant of the Kahn-Kalai Conjecture
and sharp form of Bourgain's Theorem, both of which can be thought
of as isoperimetric stability results. Both proofs closely follow
existing proofs and substitute our new hypercontractivity inequality
for the standard hypercontractivity theorem:
for the first we follow a proof of the isoperimetric inequality,
and for the second the proof of KKL given
by Bourgain and Kalai \cite{BourgainKalai19}
(their main idea is to apply the argument we gave in Theorem
\ref{thm:warm-up} for each of the derivatives of $f$).

\begin{proof}[Proof of Theorem \ref{thm:Variant of Kahn kalai}]
 We prove the contrapositive statement that for
 a sufficiently large absolute constant $C$, if $f$ is a Boolean function
 such that $\mu_p(f_{J\to 1})\le e^{-{CK}}$  for all $J$ of size at most $CK$, then $p\mathrm{I}[f]>K\mu_p(f)$.
 Let $f$ be such a function, and set $\delta=e^{-CK}$.
 Provided that $C>2$, $f$ is  $\left(2K,\delta\right)$-global, and has
 $p$-biased measure at most $\delta$.
 By Corollary \ref{lem:normsense}, we have
\[
\|f^{\le 2K}\|_2^2 \le 10^{2K}\delta^{\frac{1}{3}}\mu_p\left(f\right)\leq \mu_p\left(f\right)/2,
\]
provided that $C$ is sufficiently large. Hence,
\[
  \|f^{>2K}\|_2^2 = \|f\|_2^2-\|f^{\le 2K}\|_2^2
  \geq
  \mu_p\left(f\right)/2.
  \]
By \eqref{eq:inf} we obtain $p(1-p)\mathrm{I}[f]\ge
2K\|f^{>2K}\|_2^2$, so $p\mathrm{I}[f]> K\mu_p(f)$.
\end{proof}

Next we require the following lemma which bounds the norm of
a low degree truncation in terms of the total influence.
\begin{lem} \label{lem:normtruncate}
  Let $r \ge 0$. Suppose that for each nonempty
  set $S$ of size
  at most $r$, $\mathrm{I}_S\left(f^{\le r}\right)\le \delta.$ Then
  \[
    \|f^{\le r}\|_2^2\le \mu_p(f)^2 + 5^{r-1}\delta^{\frac{1}{3}}\sigma^2\mathrm{I}[f].
    \]
\end{lem}

\begin{proof}
Let $g_i := f_{i \to 1} - f_{i \to 0}$.
Then for each $S$ of size at most $r-1$, with $i\notin S$ we have
\[
  \mathrm{I}_S(g_{i}^{\le r})=\mathrm{I}_{S\cup\{i\}}(f^{\le r}) \le
  \delta,
\]
and for each $S$ containing $i$ we have $\mathrm{I}_S((g_i)^{\le r})=0$.
By Lemma \ref{lem:normsense0},
$\mb{E}[((g_i)^{\le r})^2]
\le 5^{r-1}\delta^{\frac{1}{3}}  \mb{E}[g_i^2]$.
The lemma now follows by summing over all $i$, using $\sum_i \mb{E}[g_i^2] = \mathrm{I}(f)$:
\begin{align*}
\| f^{\le r} \|_2^2
& = \sum_{|S| \le r} \hat{f}(S)^2
\le \hat{f}(\es)^2 + \sum_{|S| \le r} |S|\hat{f}(S)^2 \\
& = \mu_p(f)^2 + \sS^2 \sum_i \mb{E}[ ((g_i)^{\le r})^2 ]
\le \mu_p(f)^2 + 5^{r-1} \dD^{1/3} \sS^2 \mathrm{I}(f). & \qedhere
\end{align*}
\end{proof}

We now establish a variant of Bourgain's Theorem
for general Boolean functions, in which we replace
the conclusion on the measure of a restriction
by finding a large generalised influence.

\begin{thm}\label{thm:Bourgain++}
 Let $f\colon\{0,1\}^n\to\{0,1\}.$ Suppose that
$p\mathrm{I}[f] \le
  K \mu_p\left(f \right)(1-\mu_p(f))$. Then there exists
  an $S$ of size $2K$, such that $\mathrm{I}_S(f)\ge 5^{-8K}.$
\end{thm}
\begin{proof}
Let $r = 2K$ and let
$\dD = 5^{-8K}$.
Suppose for contradiction that $\mathrm{I}_S(f)\le\delta$ for each
set $S$ of size at most $r$. By Lemma \ref{lem:normtruncate},
\[
  \| f^{\le r} \|_2^2- \mu_p(f)^2
\le 5^{r-1} \dD^{1/3} \sS^2 I(f)
< p\mathrm{I}[f]/2K\le \mu_p(f)(1-\mu_p(f))/2.
\]

On the other hand, by Parseval \[
  \| f - f^{\le r} \|_2^2
= \sum_{|S| \ge r} \hat{f}(S)^2
\le r^{-1} \sum_{|S| \ge r} |S| \hat{f}(S)^2
\le r^{-1} p(1-p) \mathrm{I}(f)\le
\mu_p(f)(1-\mu_p(f))/2.
\]
However, these bounds contradict the fact that
\begin{align*}
  \mu_p(f)(1-\mu_p(f))=\|f\|_2^2 -\mu_p(f)^2= \| f^{\le r} \|_2^2-
  \mu_p(f)^2 + \| f - f^{\le r} \|_2^2. & \qedhere
\end{align*}
 \end{proof}

\begin{proof}[Proof of Theorem \ref{thm:Bourgain+}]
The theorem follows immediately from Theorem \ref{thm:Bourgain++} and
Lemma \ref{rem}.
\end{proof}

\subsection{Sharpness examples}
We now give two examples showing
sharpness of the theorems in this section,
both based on the tribes function of Ben-Or
and Linial \cite{ben1990collective}.

\begin{example} \label{eg1}
We consider the anti-tribes function $f=f_{s,w}:\{0,1\}^n \to \{0,1\}$
defined by $s$ disjoint sets $T_1,\dots,T_s \sub [n]$ each of size $w$,
where $f(x) = \prod_{j=1}^s \max_{i \in T_j} x_i$,
i.e.\ $f(x)=1$ if for every $j$ we have $x_i=1$ for some $i \in T_j$,
otherwise $f(x)=0$. We have $\mu_p(f) = (1-(1-p)^w)^s$
and $\mathrm{I}[f] = \mu_p(f)' = sw(1-p)^{w-1} (1-(1-p)^w)^{s-1}$.
We choose $s,w$ with $s(1-p)^w = 1$
(ignoring the rounding to integers) so that
$\mu_p(f)=(1-s^{-1})^s$ is bounded away from $0$ and $1$,
and $K = (1-p)p\mathrm{I}[f] = pw(1-s^{-1})^{-1} \mu_p(f) = \Tt(pw)$.
Thus $\log s = w\log (1-p)^{-1} = \Tt(K)$.
However, for any $J \sub [n]$ with $|J| = t \le s$ we have
$\mu_p(f_{J\to1}) \le (1-s^{-1})^{s-t} \le 2^{t/s} \mu_p(f)$,
so to obtain a density bump of $e^{-o(K)}$
we need $t = e^{-o(K)} s = e^{\Oo(K)} \gg K$.
Thus Theorem \ref{thm:Bourgain+} is sharp.
\end{example}

\begin{example} \label{eg2}
Let $f(x) = f_{s,w}(x) \prod_{i \in T} x_i$
with $f_{s,w}$ as in Example \ref{eg1}
and $T \sub [n]$ a set of size $t$ disjoint from $\cup_j T_j$.
We have $\mu_p(f) = p^t (1-(1-p)^w)^s$ and
$I[f] = \mu_p(f)' = tp^{t-1} (1-(1-p)^w)^s
+ p^t sw(1-p)^{w-1} (1-(1-p)^w)^{s-1}$.
We fix $K>1$ and choose $s,w$ with $s(1-p)^w = K$, so that
$\mu_p(f) = p^t (1-K/s)^s = p^t e^{-\Tt(K)}$ for $s>2K$
and $p(1-p)I[f] = \mu_p(f) ( (1-p)t + pw K (1-K/s)^{-1} )
= \mu_p(f) \Tt(K) $ if $pw = \Tt(1)$ and $t=O(K)$.
For any $J \sub [n]$ with $|J|=t+u \le t+s$
we have  $\mu_p(f_{J\to1}) \le (1-K/s)^{s-u}
\le e^{-K(1-u/s)} \le e^{-K/2}$ unless $u > s/2 = \Tt(K)$.
Thus Theorem \ref{thm:Variant of Kahn kalai} is sharp.
\end{example}

\subsection{Sharp thresholds: the traditional approach}\label{subsec:ST2}
In this section we deduce Theorem \ref{thm:Sharp threshold result}
from our edge-isoperimetric stability results
and the Margulis--Russo Lemma.
Recall that a monotone Boolean function
is $M$-global in an interval if
$\mu_p\left(f_{J\to 1}\right)\le \mu_p\left(f\right)^{0.01}$
for each $p$ in the interval and set $J$ of size $M$.
We prove the following slightly stronger version of Theorem
\ref{thm:Sharp threshold result}.
\begin{thm} \label{thm:trad}
 There exists an absolute constant $C$ such that
 the following holds for any monotone Boolean function $f$
 that is $M$-global in some interval  $\left[p,q\right]$:
 if $q\le p_c$ and  $\mu_p\left(f\right)\ge e^{-M/C}$ then
  \begin{equation} \label{eq:trad}
    \mu_q\left(f \right) \ge \mu_p(f)^{\left(\frac{p}{q}\right)^{1/C}}.
  \end{equation}
 In particular, $q\le M^{C} p$.
\end{thm}
\begin{proof}
 By Theorem \ref{thm:Variant of Kahn kalai}, since $f$ is $M$-global throughout the interval, there exists a constant
 $C$ such that  $\mathrm{I}_{x}\left[f\right]
 \ge \frac{\mu_{x}(f) \log(\frac{1}{\mu_{x}(f)})}{Cx}$
 for all $x$ in
 the interval $\left[p,q\right]$.
 By the Margulis-Russo lemma,
  \[
   \frac{d}{dx} \log \left(-\log
   (\mu_{x}\left(f\right))\right)
   =\frac{\mu_{x}(f)'}{\mu_{x}(f)\log(\mu_{x}\left(f\right))}
   =\frac{I_{x}[f]}{\mu_{x}(f)\log(\mu_{x}\left(f\right))}
   \le \frac{-1}{Cx}
 \]
 in all of the interval $\left[p,q\right]$. Hence,
\[
 \log\left(-\log
 (\mu_q(f))\right)\le \log(-\log(\mu_p(f))) - \frac{\log(\frac{q}{p})}{C}.
\]
The first part of the theorem follows by taking exponentials,
multiplying by $-1$ then taking exponentials again.
To see the final statement, note that $q \le p_c$
implies $\mu_q\left(f\right)\le\frac{1}{2}$.
We cannot have $q \ge M^c p$,
as then the right hand side in \eqref{eq:trad}
would be larger than $e^{-\frac{1}{C}} > 1/2$ for large $C$.
To obtain Theorem \ref{thm:Sharp threshold result}
we substitute $q=p_c$.
  \end{proof}

\section{Noise sensitivity and sharp thresholds} \label{sec:noise}

We start this section by showing that sparse global functions
are noise sensitive; Theorem \ref{thm:Noise sensitivity}
follows immediately from
Theorem \ref{thm:quantitative noise sensitivity}.

\begin{thm}\label{thm:quantitative noise sensitivity}
 Let $\rho\in\left(0,1\right)$, and let $\epsilon>0$. Let
 $r=\frac{\log(2/\epsilon)}{\log(1/\rho)}$,
 and let $\delta = 10^{-3r-1}\epsilon^3$.
 Suppose that $f$ is an
  $\left(r,\delta\right)$-global Boolean
  function with $\mu_p\left(f\right)<\delta$.  Then
  \[
    \mathrm{Stab}_\rho\left(f\right)\le \epsilon \mu_p\left(f\right).
    \]
  \end{thm}
\begin{proof}
We have
\begin{align*}
\left\langle \mathrm{T}_{\rho}f,f\right\rangle
& \le \sum_{\left|S\right|\le r}\hat{f}\left(S\right)^{2}
+ \rho^{r}\sum_{\left|S\right|>r}\hat{f}\left(S\right)^{2}
\le \mathbb{E}\left[\left(f^{\le r}\right)^{2}\right]
+ \frac{\eps}{2}\mu_{p}(f).
\end{align*}
The statement now follows from Corollary \ref{lem:normsense},
which gives $\mb{E}[(f^{\le r})^2]
\le 10^{r} \delta^{1/3} \mb{E}[f^2]
< \eps \mu_p(f)/2$.
\end{proof}

In the remainder of this section,
following \cite{lifshitz2018hypergraph},
we deduce sharp thresholds from noise sensitivity
via the following \emph{directed noise operator},
which is implicit in the work of
Ahlberg, Broman, Griffiths and Morris \cite{ahlberg2014noise}
and later studied in its own right by
Abdullah and Venkatasubramanian \cite{abdullah2015directed}.
\begin{defn}
Let $D\left(p,q\right)$
denote the unique distribution on pairs
$\left(\boldsymbol{x,y}\right)\in
\left\{ 0,1\right\} ^{n}\times\left\{ 0,1\right\} ^{n}$
such that $\xb\sim\mu_{p}$, $\yb\sim\mu_{q}$,
all $\xb_{i}\le\yb_{i}$ and
$\{ (\xb_{i},\yb_{i}) : i \in [n]\}$ are independent.
We define a linear operator
$\mathrm{T}^{p \to q}: L^2(\{0,1\}^n, \mu _p) \to  L^2(\{0,1\}^n, \mu _q)$ by
\[ \mathrm{T}^{p\to q}\left(f\right)\left(y\right)
=\mathbb{E}_{(\xb,\yb)\sim D\left(p,q\right)}
\left[f\left(\xb\right)|\,\yb=y\right]. \]
\end{defn}

The directed noise operator $\mathrm{T}^{p \to q}$
is a version of the noise operator
where bits can be flipped only from $0$ to $1$.
The associated notion of directed noise stability,
i.e.\ $\left\langle f,\art f\right\rangle _{\mu_{q}}$,
is intuitively a measure of how close
a Boolean function $f$ is to being monotone.
Indeed, for any $(\mathbf{x},\mathbf{y})$ with all $x_i \le y_i$
we have $f\left(\xb\right)f\left(\yb\right)\le f\left(\xb\right)$,
with equality if $f$ is monotone, so
\[ \left\langle f,\art f\right\rangle
= \mathbb{E}_{(\mathbf{x},\mathbf{y})\sim D\left(p,q\right)}
  \left[f\left(\xb\right)f\left(\yb\right)\right]
\le \mathbb{E}_{(\mathbf{x},\mathbf{y})
\sim D\left(p,q\right)}\left[f\left(\xb\right)\right]
=\mu_{p}\left(f\right),\]
with equality if $f$ is monotone\footnote{
The starting point for \cite{lifshitz2018hypergraph}
is the observation that this inequality is close
to an equality if $f$ is almost monotone.}.
We note that the adjoint operator
$\left(\mathrm{T}^{p\to q}\right)^{\star}:
L^2(\{0,1\}^n, \mu _q) \to  L^2(\{0,1\}^n, \mu _p)$
defined by $\bgen{\mathrm{T}^{p\to q}f,g}
= \bgen{f,\left(\mathrm{T}^{p\to q}\right)^{\star}g} $
 satisfies $\left(\mathrm{T}^{p\to q}\right)^{\star}
= \mathrm{T}^{q\to p}$, where
\[ \mathrm{T}^{q\to p}\left(g\right)\left(x\right)
=\mathbb{E}_{(\xb,\yb)\sim D\left(p,q\right)}
\left[g\left(\yb\right)|\,\xb=x\right]. \]
The following simple calculation relates these
operators to the noise operator.

\begin{lem} \label{calcrho}
Let $0<p<q<1$ and $\rho = \frac {p(1-q)}{q(1-p)}$. Then
$\left(\mathrm{T}^{p\to q}\right)^{\star}\mathrm{T}^{p\to q}
= \mathrm{T}_{\rho}$ on $L^2(\{0,1\}^n, \mu _p)$.
\end{lem}

\begin{proof}
We need to show that the following distributions
on pairs of $p$-biased bits $(\mathbf{x},\mathbf{x'})$ are identical:
(a) let $\mathbf{x}$ be a $p$-biased bit, with probability $\rho$
let $\mathbf{x'}=\mathbf{x}$,
otherwise let $\mathbf{x'}$ be an independent $p$-biased bit,
(b) let $(\mathbf{x},\mathbf{y}) \sim D(p,q)$ and then
$(\mathbf{x'},\mathbf{y}) \sim D(p,q) \mid y$.
It suffices to show $\mb{P}(x \ne x')$
is the same in both distributions.
We condition on $x$. Consider $x=1$. In distribution (a)
we have $\mb{P}(\mathbf{x'}=0) = (1-\rho)(1-p)$.
In distribution (b) we have $\mb{P}(\mathbf{y}=1)=1$
and then $\mb{P}(\mathbf{x'}=0) = 1-p/q = (1-\rho)(1-p)$, as required.
Now consider $\mathbf{x}=0$. In distribution (a)
we have $\mb{P}(\mathbf{x'}=1) = (1-\rho)p$.
In distribution (b) we have $\mb{P}(\mathbf{y}=1)=\tfrac{q-p}{1-p}$
and then $\mb{P}(\mathbf{x'}=1 \mid \mathbf{y}=1) = p/q$, so
$\mb{P}(\mathbf{x'}=1) = \tfrac{p(q-p)}{q(1-p)} = (1-\rho)p$, as required.
\end{proof}

We now give an alternative way to deduce sharp threshold results,
using noise sensitivity, rather than the traditional approach
via total influence (as in the proof of Theorem \ref{thm:trad}).
Our alternative approach has the following additional
nice features, both of which have been found useful in
Extremal Combinatorics (see \cite{lifshitz2018hypergraph}).
\begin{enumerate}
\item To deduce a sharp threshold result in an interval
$\left[p,q\right]$ it is enough to show that $f$ is
global only according to the $p$-biased distribution.
This is a milder condition than the one in the traditional approach,
that requires globalness throughout the entire interval.
\item The monotonicity requirement may be relaxed
to ``almost monotonicity''.
\end{enumerate}

\begin{prop}
\label{prop:Noise sensetivity implies sharp threshold}
Let $f\colon\left\{ 0,1\right\} ^{n}\to\left\{ 0,1\right\} $
be a monotone Boolean function. Let $0<p<q<1$
and $\rho=\frac{p\left(1-q\right)}{q\left(1-p\right)}$.
Then $\mu_q(f) \ge \mu_p(f)^2 / \mathrm{Stab}_{\rho}\left(f\right)$.
\end{prop}
\begin{proof}
By Cauchy\textendash Schwarz and Lemma \ref{calcrho},
\begin{align*}
\mu_{p}\left(f\right)^2=\left\langle \art f,f\right\rangle _{\mu_{q}}^2\le \left\langle \art f,\art f\right\rangle _{\mu_{q}}\left\langle f,f\right\rangle _{\mu_{q}}
=\left\langle \mathrm{T}_{\rho}f,f\right\rangle _{\mu_{p}}\mu_{q}\left(f\right).
& \qedhere
\end{align*}
\end{proof}

The above proof works not only for monotone functions,
but also for functions where the first equality above
is replaced by approximate equality (which is a natural
notion for a function to be ``almost monotone'').
The following sharp threshold theorem for global functions
is immediate from Theorem \ref{thm:quantitative noise sensitivity}
and Proposition \ref{prop:Noise sensetivity implies sharp threshold}.

\begin{thm}
\label{thm:global sharp threshold}
For any $\zeta>0$ there is $C_0>1$ so that
for any $\eps,p,q \in (0,1/2)$ with $q\ge (1+\zeta)p$
and $C>C_0$, writing $r=C\log \eps^{-1}$ and $\dD=C^{-r}$,
any monotone $(r,\delta)$-global Boolean function $f$ whose $p$-biased
measure is at most $\delta$
satisfies $\mu_q(f) \ge \eps^{-1} \mu_p(f)$.
\end{thm}

\section{General hypercontractivity} \label{sec:genhyp}

In this section we generalise Theorem \ref{thm:Hypercontractivity} in two
different directions. One direction is showing hypercontractivity
from general $q$-norms to the $2$-norm (rather than merely
treating the case $q=4$); the other is replacing the cube
by general product spaces.

\subsection{Hypercontractivity with general norms}
\label{subsec:multilinear}

We start by describing a more convenient general setting
in which we replace characters on the cube
by arbitrary random variables.
To motivate this setting, we remark that one can extend
the proof of Theorem \ref{thm:hypref} to any
random variable of the form
\begin{equation} \label{eq:fRV}
f  = \sum_{S\subset [n]} a_S\prod_{i\in S} \mathbf{Z}_i,
\end{equation}
where $\mathbf{Z}_1,\ldots,\mathbf{Z}_n$
are independent real-valued random variables having
expectation $0$, variance $1$ and $4$th moment at most
$\sigma^{-2}$. To motivate the analogous setting for general $q>2$,
we note that the characters $\chi_i^p$ satisfy
\[
  \mathbb{E}[|\chi_i^p|^q]\le
  \|\chi_i^p\|_\infty^{q-2}\|\chi_i^p\|_2^2 = \sigma^{2-q}.
  \]
This suggests replacing the $4$th moment condition by
$\|\mathbf{Z}_i\|_q^q\le \sigma^{2-q}$.
Given $f$ as in \eqref{eq:fRV}, we define
the (generalised) derivatives by substituting
the random variables $Z_i$ for the characters $\chi_i^p$
in our earlier Fourier formulas, i.e.\
\[
  \mathrm{D}_i[f]=\sum_{S:\,i\in S}a_S\prod_{j\in S\setminus \{i\}}
  \mathbf{Z}_i
  \quad \text{and} \quad
  \mathrm{D}_T(f)=\sum_{S:\, T\subset
    S}a_S\prod_{j\in S\backslash T}\mathbf{Z}_i,
\]
Similarly, we adopt analogous definitions of the
generalised influences and noise operator, i.e.\
 \[
  \mathrm{I}_S[f]=
  \|\frac{1}{\sigma}\mathrm{D}_S[f]\|_2^2
  \quad \text{and} \quad
  \mathrm{T}_\rho[f]=\sum_S \rho^{|S|}a_S\prod_{i\in S}{\mathbf{Z}_i}.
\]

We prove the following hypercontractive inequality.
\begin{thm}\label{thm:qth moment}
  Let $q \ge 2$ and $\mathbf{Z}_1,\ldots,\mathbf{Z}_n$
  be independent real-valued  random variables satisfying
  \[
    \mathbb{E}[\mathbf{Z}_i]=0, \quad  \mathbb{E}[\mathbf{Z}_i^2]=1,
    \quad \text{and} \quad \mathbb{E} [\left| \mathbf{Z}_i \right|^q] \le \sigma^{2-q}.
  \]
  Let $f=\sum_{S\subset [n]}  a_S\prod_{i\in S}\mathbf{Z}_i$ and
  $\rho< \frac{1}{2q^{1.5}} $. Then
  \[
    \| \mathrm{T}_{\rho}f \|_q^q \le \sum_{S\subset [n]}
    \sigma^{(2-q)|S|}\| \mathrm{D}_S(f) \|_2^{q}.
    \]
  \end{thm}

  Theorem \ref{thm:qth moment} is a qualitative generalisation
  of Theorem \ref{thm:hypref} (with smaller $\rho$, which we do not
  attempt to optimise). The following generalised variant of
  Theorem \ref{thm:Hypercontractivity} follows by repeating
  the proof in Section \ref{sec:hyp}.
  \begin{thm}
    Let $q>2$, let $f=\sum_{S \sub [n]}a_S\prod_{i\in S}\mathbf{Z_i}$ let
    $\delta>0$, and let $\rho \le (2q)^{-1.5}$. Suppose
    that $\mathrm{I}_S[f] \le \beta \|f\|_2^2$ for all $S\subset
    [n]$. Then
    \[
      \| \mathrm{T}_{\rho} [f] \|_q \le \beta^{\frac{q-2}{2q}} \|f\|_2.
      \]
    \end{thm}

We now begin with the ingredients of the
proof of Theorem \ref{thm:qth moment},
following that of Theorem \ref{thm:hypref}.
For $0 \le t \le n$ let
\[ f_t = \sum_S a_S \chi_S^t, \ \ \text{ where }
\chi^t_S = \prod\limits_{i\in S\cap [t]}{\chi^{1/2}_i}
  \prod\limits_{i\in S\setminus [t]}{\mathbf{Z}_i}. \]
Here, just as in Section \ref{sec:hyp},
the function $f_t$ interpolates from the original function
$f_0=f$ to $f_n = \sum_S a_S \chi^{1/2}_S
\in L^2(\{0,1\}^n,\mu_{1/2})$.
As $\{ \chi^t_S: S \sub [n]\}$ are orthonormal
we have $\|f_t\|_2 = \|f\|_2$ for all $t$.

As before, we define the noise operators
$\mathrm{T}^{t}_{\rho',\rho}$ on a function
$f = \sum_S a_S\chi_S^t$ by
\[ \mathrm{T}^t[f]
=\sum_S\rho'^{|S \cap [t]|}\rho^{|S\setminus [t]|}a_S\chi_S^t.\]
Thus $\mathrm{T}^t_{\rho',\rho}$ interpolates from
$\mathrm{T}^0_{\rho',\rho}=\mathrm{T}_{\rho}$ (for the original function) to
$\mathrm{T}^n_{\rho',\rho}=\mathrm{T}_{\rho'}$ (for $\mu_{1/2}$).

Our goal will now be to adjust Lemma \ref{lem:Replacement step}
to the general setting, which is similar in spirit to the 4-norm case,
although somewhat trickier. It turns out that the case $n=1$ poses the
main new difficulties, so we start with this in the next lemma.

\begin{lem} \label{lem:n=1}
Let $q>2$ and
$\mathbf{Z}$ be a random variable satisfying $\mathbb{E}[\mathbf{Z}]=0,
\mathbb{E}[\mathbf{Z}^2]=1, \mathbb{E}[|\mathbf{Z}|^q]\le \sigma^{2-q}.$
Let $e,d\in \mathbb{R}$ and $\rho\in(0,\frac{1}{2q})$.
Then $\|e + \rho d\mathbf{Z}\|_q^q \le \|e + d
\chi^{\frac{1}{2}}\|_q^q+\sigma^{2-q}d^q$.
\end{lem}

\begin{proof}
  If $e = 0$ then the lemma is trivial.
  Therefore we may rescale and assume that $e=1$.
  It will be convenient to consider both sides of the inequality
  as functions of $d$: we write
  \[ f(d) =  \|1+\rho d\mathbf{Z}\|_q^q
   \quad \text{and} \quad
  g(d) = \|1 + d  \chi^{\frac{1}{2}}\|_q^q+\sigma^{2-q}d.
    \]
  As $f(0)=g(0)$, it suffices to show that $f'(0)=g'(0)$
  and $f''\le g''$ everywhere.

  Let us compute the derivatives. We note that the function
  $x\mapsto |x^q|$ has derivative $q|x|^{q-1}\mathrm{sign}(x)$,
  which is in turn continuously differentiable for $q>2$. Thus
  \begin{align*}
f' & =
\mathbb{E}[q\left|1+\rho d\mathbf{Z}\right|^{q-1}\mathrm{sign}(1+\rho d\mathbf{Z})\rho \mathbf{Z}]=
 \rho q\mathbb{E}[|1+\rho d\mathbf{Z}|^{q-1}\mathrm{sign}(1+\rho d\mathbf{Z}) \mathbf{Z}] \ \ \text{ and } \\
 f'' & = (q-1)q\rho^2 \mathbb{E}[|1+\rho d\mathbf{Z}|^{q-2} \mathbf{Z}^2].
\end{align*}
    Differentiating $g$ we obtain
  \begin{align*}
g' & =q\mathbb{E}\Big[\left|1 + d
  \chi^{\frac{1}{2}}\right|^{q-1}\mathrm{sign}(1+d\chi^{\frac{1}{2}})\chi^{\frac{1}{2}}\Big]+q\sigma^{2-q}d^{q-1}
\ \ \text{ and } \\
g'' & = q(q-1) \mathbb{E}\Big[\left| 1+d\chi^{\frac{1}{2}} \right|
      ^{q-2}\left(\chi^{\frac{1}{2}}\right) ^2 \Big]
      +q(q-1)d^{q-2}\sigma^{2-q} \ge
      q(q-1)/2 + q(q-1)d^{q-2}\sigma^{2-q}.
\end{align*}
Thus $g'(0)=f'(0)=0$ and it remains to show $f''\le g''$ everywhere.
Our strategy for bounding $f''$ is to decompose the expectation over
two complementary events $E_1$ and $E_2$,
where $E_1$ is the event that
$|1+ \rho d \mathbf{Z}| \le |d \mathbf{Z}|$
(and $E_2$ is its complementary event).
We write $f''=f''_1 +f''_2$, where each
\[ f''_i = (q-1)q\rho^2 \mathbb{E}[|1+\rho d\mathbf{Z}|^{q-2} \mathbf{Z}^2
       \mathbf{1}_{E_i}].        \]
First we note the bound
\[
f''_1 \le q(q-1) \rho^2 d^{q-2}\mathbb{E}[|\mathbf{Z}|^q]
\le q(q-1) d^{q-2}\sigma^{2-q}.
\]
Given the above lower bound on $g''$,
it remains to show $f''_2\le q(q-1)/2$.
On the event $E_2$ we have
\[
  |d \mathbf{Z}| \le |1+ \rho d \mathbf{Z}|\le 1 +|\rho d \mathbf{Z}|.
\]
Rearranging, we obtain $|\rho d \mathbf{Z}|(\rho^{-1}-1) \le 1.$
Since $\rho^{-1}\geq 2q$, we get
\[
  1+ |\rho d \mathbf{Z}| \le
 1 + \frac{1}{2q-1}.
  \]
    Using $\mathbb{E}[\mathbf{Z}^2]=1$ this yields
    \[
      f''_2
      \le q(q-1)\rho^2 \Big (1+\frac{1}{2q-1}\Big )^{q-2}\le
      e\rho^2q(q-1)\le q(q-1)/2
      .
    \]
    Hence $f''=f''_1 + f''_2 \le g''$ for any value of $d$.
    This completes the proof of the lemma.
    \end{proof}

We are now ready to show the replacement step.

\begin{lem} \label{lem:q-Replacement step}
$\mb{E}[(\mathrm{T}^{t-1}_{2q\rho,\rho} f_{t-1})^q]
\le \mb{E}[(\mathrm{T}^t_{2q\rho,\rho} f_t)^q]
+ \sigma^{2-q} \mb{E}[(\mathrm{T}^t_{2q\rho,\rho} ((\mathrm{D}_t f)_t) )^q]$.
\end{lem}

\begin{proof}
We write
\begin{align*}
f_t & = \chi^{1/2}_t g + h \ \ \text{ and } \ \
 f_{t-1} =  \chi^p_t g + h, \ \ \text{ where } \\
g & = (\mathrm{D}_t f)_t
= \sum_{S: t \in S} \hat{f}(S) \chi^t_{S \sm \{t\}}
= \sum_{S: t \in S} \hat{f}(S) \chi^{t-1}_{S \sm \{t\}}
= (\mathrm{D}_t f)_{t-1}, \ \ \text{ and } \\
h & = \mb{E}_{x_t \sim \mu_{1/2}} f_t
= \sum_{S: t \notin S} \hat{f}(S) \chi^t_S
= \sum_{S: t \notin S} \hat{f}(S) \chi^{t-1}_S
= \mb{E}_{\mathbf{Z}_t} f_{t-1}.
\end{align*}
We also write
\begin{align*}
\mathrm{T}^t_{2q\rho,\rho} f_t & = 2q\rho \chi^{1/2}_t d + e \ \ \text{ and } \ \
\mathrm{T}^{t-1}_{2q\rho,\rho} f_{t-1} = \rho \mathbf{Z}_t d + e, \ \ \text{ where } \\
d & = \mathrm{T}^t_{2q\rho,\rho} g = \mathrm{T}^{t-1}_{2q\rho,\rho} g
\ \ \text{ and } \ \
e = \mathrm{T}^t_{2q\rho,\rho} h = \mathrm{T}^{t-1}_{2q\rho,\rho} h.
\end{align*}
As before, we can calculate the expectations in the statement of the lemma
by conditioning on all coordinates other than $\mathbf{Z}_t$ and
$\chi_t^{\frac{1}{2}}$, so the lemma follows from Lemma
\ref{lem:n=1}, with $2qd$ in place of $d$.
\end{proof}

From now on, everything is similar to Section \ref{sec:hyp}.
We may apply the previous lemma inductively to obtain.

\begin{lem} \label{qhypinduct}
$\| \mathrm{T}^i_{2q\rho,\rho} f_i \|_q^q  \le \sum_{S \sub [n] \sm [i]}
\sigma^{(2-q)|S|}\| \mathrm{T}^n_{2 q\rho,\rho} ((\mathrm{D}_S f)_n) \|_q^q$
 for all $0 \le i \le n$.
\end{lem}

In particular, recalling that
$\mathrm{T}^0_{2 q \rho,\rho}=\mathrm{T}_{\rho}$ on the original function and
$\mathrm{T}^n_{2 q \rho,\rho}=\mathrm{T}_{2 q \rho}$ on $\mu_{1/2}$,
the case $i=0$ of Lemma \ref{qhypinduct} is as follows.

\begin{prop}
\label{prop:q-Reason for hypercontractivity}
$\| \mathrm{T}_{\rho} f \|_q^q  \le \sum_{S \sub [n]}
 \sigma^{(2-q)|S|} \| \mathrm{T}_{2 q \rho} ((\mathrm{D}_S f)_n) \|_q^q$.
\end{prop}

The $q$-norms on the right hand side of
Proposition \ref{prop:q-Reason for hypercontractivity}
are with respect to the uniform measure $\mu_{1/2}$,
where we can apply standard hypercontractivity with noise rate $ \le 1/\sqrt{q-1}$
to obtain \[
  \| \mathrm{T}_{2 q \rho} ((\mathrm{D}_S f)_n) \|_q^q
\le \| (\mathrm{D}_S f)_n \|_2^q = \| \mathrm{D}_S f \|_2^q.
\]

This completes the proof of Theorem \ref{thm:qth moment}.

In the case where the $\mathbf{Z}_i$ have different $q$th
moments, the proof can be adjusted to give a better upper bound.
We write
\begin{equation} \label{eq:Zgen}
\mathbb{E}[\mathbf{Z}_i^q] = \sigma_i^{2-q}, \quad
\sigma_S=\prod_{i\in S}\sigma_i
\ \ \text{ and } \ \
\mathrm{I}_S[f]=\|\frac{1}{\sigma_S}\mathrm{D}_S[f]\|_2^2.
\end{equation}
The proof of Theorem \ref{thm:qth moment}
yields the following variant of Theorem \ref{thm:hypref}.

\begin{thm}\label{thm:qth moment with different norms}
  Let $q\ge 2$, let $\rho\le (2q)^{-1.5}$, and let $f =\sum
  a_S\prod_{i\in S}\mathbf{Z}_i$ with $Z_i$ as in \eqref{eq:Zgen}. Then
  \[
    \|\mathrm{T}_{\rho}f\|_q^q \le \sum
    _{S\subset [n]}\sigma_S^{2-q}\|\mathrm{D}_S[f]\|_2^q.
    \]
  \end{thm}

  The following variant of Theorem \ref{thm:Hypercontractivity}
  follows from Theorem \ref{thm:qth moment with different norms}. The proof is
  similar to the one given in Section \ref{sec:hyp}, where Theorem
  \ref{thm:Hypercontractivity} is deduced from Theorem \ref{thm:hypref}.

  \begin{thm}\label{thm:Hypercontractivity with different norms}
Let $q>2$, $\beta>0$ and $\rho \le (2q)^{-1.5}$.
Suppose $f=\sum_{S \sub [n]}a_S\prod_{i\in S}\mathbf{Z}_i$
 with $Z_i$ as in \eqref{eq:Zgen}
has $\mathrm{I}_S[f] \le \beta \|f\|_2^2$
for all $S\subset    [n]$. Then
    \[
      \| \mathrm{T}_{\rho}f \|_q\le \beta^{\frac{q-2}{2q}}\|f\|_2.
      \]
    \end{thm}

    Finally, we state the following variant of Lemma
    \ref{lem:applying_hypercontractivity}, which is easy to deduce
    from Theorem \ref{thm:Hypercontractivity with different norms}.

    \begin{lem}\label{lem:q-applying_hypercontractivity}
     Let $q>2$ and $\delta>0$.
     Suppose $f=\sum_{S \sub [n]}a_S\prod_{i\in S}\mathbf{Z}_i$
   with $Z_i$ as in \eqref{eq:Zgen}
   has $\mathrm{I}_S[f] \le \delta $ for
    all $|S| \le r$. Then
    \[
      \| f \|_q \le (2q)^{1.5 r}\delta^{\frac{q-2}{2q}}\|f\|_2^{\frac{2}{q}}.
      \]
      \end{lem}

\subsection{A hypercontractive inequality for product spaces}
Now we consider the setting of a general discrete
product space $(\Oo,\nu) = \prod_{t=1}^n (\Oo_t,\nu_t)$.
We assume
$p_t=\min_{\oO_t \in \Oo_t} \nu_t(\oO_t) \in (0,1/2)$
for each $t \in [n]$, and we write $p=\min_t{p_t}$.
We recall the projections $\mathrm{E}_J$ on $L^2(\Oo, \nu)$ defined
by $(\mathrm{E}_J f)(\oO) = \mb{E}_{\oO_J}[ f(\oO) \mid \oO_{\ov{J}} ]$,
the generalised Laplacians $\mathrm{L}_S$
defined by composing $\mathrm{L}_t$ for all $t \in S$,
where $\mathrm{L}_t f = f - \mathrm{E}_t f$,
and the generalised influences
$\mathrm{I}_S(f) = \mb{E}[\mathrm{L}_S(f)^2] \prod_{i \in S} \sS_i^{-2}$,
where $\sS_i^2=p_i(1-p_i)$.

We will require the theory of orthogonal decompositions
in product spaces, which we summarise following the
exposition in \cite[Section 8.3]{o2014analysis}.
For $f\in L^2(\Oo,\nu)$ and $J,S \sub [n]$
we write $f^{\sub J} = \mathrm{E}_{\ov{J}} f $ and define
$f^{=S} = \sum_{J \sub S} (-1)^{|S \sm J|} f^{\sub J}$
(inclusion-exclusion for
$f^{\sub J} = \sum_{S \sub J} f^{=S}$).
This decomposition is known as the Efron--Stein decomposition
\cite{efron1981jackknife}. The key properties of $f^{=S}$ are that it only depends
on coordinates in $S$ and it is orthogonal to any function
that depends only on some set of coordinates not containing $S$;
in particular, $f^{=S}$ and $f^{=S'}$ are orthogonal for $S \ne S'$.
We note that $f = f^{\sub [n]} = \sum_S f^{=S}$.
We have similar Plancherel / Parseval relations
as for Fourier decompositions, namely
$\bgen{f,g} =  \sum_S f^{=S} g^{=S}$,
so $\mb{E}[f^2] = \sum_S (f^{=S})^2$.

Our goal in this section is to prove an hypercontractive inequality
for the Efron--Stein decomposition in the spirit of Theorem
\ref{thm:hypref}. The noise operator is defined by
$\mathrm{T}_\rho[f]=\sum_{S\sub [n]}\rho^{|S|}f^{=S}$. It also has a
combinatorial interpretation, which is similar to the usual one on the
$p$-biased setting. Given $x\in\Omega$, a sample $\mathbf{y}\sim
N_\rho(x)$ is chosen by
independently setting $y_i$ to $x_i$ with probability $\rho$ and
resampling it from $(\Omega_i,\nu_i)$ with probability $1-\rho$.
In the general product space setting there are no good analogs to
$\mathrm{D}_i[f]$ and $\mathrm{D}_S(f)$, and we
instead work with the Laplacians,
which have similar Fourier formulas:
$\mathrm{L}_i[f]=\sum_{S:\,i\in S}f^{=S}$, and $\mathrm{L}_T[f]=\sum
_{S:\,T\subset S}f^{=S}$. In the special case where $\Omega_i=\{0,1\}$
we have $\|\mathrm{L}_S[f]\|_2 =\|\mathrm{D}_S[f]\|_2$. It will be
convenient to write
$\sigma_S=\prod_{i\in S} \sigma_i$.

The main result of this section is the following theorem.
\begin{thm}\label{thm:es}
Let $f\in L^2(\Omega, \nu )$, let $q>2$ be an even integer, and let $\rho \le
\frac{1}{8q^{1.5}}$. Then
  \[
    \| \mathrm{T}_{\rho}f \|^q_q  \le
    \sum_{S\sub [n]} \sigma_S^{2-q}
    \| \mathrm{L}_S[f] \|_2^q.
    \]
  \end{thm}

 The idea of the proof is as follows.
 We encode our function $f\in L^2(\Omega,\nu)$
 as a function $\tilde{f}:=\sum_S\|f^{=S}\|_2\chi_S$
 for appropriate $\chi_S=\prod_{i\in S}\chi_i$
 (in fact, these will be biased characters on the cube).
 We then bound $\| \mathrm{T}_\rho f\|_q $
 by $\| \mathrm{T}_\rho \tilde{f} \|_q$
 and use Theorem \ref{thm:Hypercontractivity with
   different norms} to bound the latter norm.

The main technical component of the theorem
is the following proposition.

\begin{prop} \label{prop:reduction}
  Let $g\in L^2(\Omega, \nu)$ let $\chi_S=\prod_{i\in S}\chi_i$, where
  $\chi_i$ are independent random variables having expectation $0$,
  variance $1$, and
  satisfying $\mathbb{E}[\chi_S^j]\ge \sigma_S^{2-j}$ for each integer
  $j\in \left(2, q \right]$. Let $\tilde{g}=\sum_{S\subset [n]} \|g^{=S}\|_2 \chi_S$.
  Then
  \[
    \|g\|_q\leq \|\tilde{g}\|_q.
    \]
  \end{prop}

  Below, we fix $\chi_S$ as in the proposition, and
  let $\tilde{\circ}$ denote the operator mapping a function
  $g\in L^2(\Omega, \nu)$ to the function
  $\sum_{S\subset[n]}g^{=S}\chi_S$.

To prove the proposition, we will expand out
$\|g\|_q^q$ and $\|\tilde{g}\|_q^q$ according to their definitions
and compare similar terms: namely, we show that a term of the form
$\mathbb{E}[\prod_{i=1}^q g^{=S_i}]$ is bounded by the corresponding term
in $\|\tilde{g}\|_q^q$, i.e. $\prod^q_{i=1}\|g^{=S_i}\|_2\mathbb{E}[\prod_{i=1}^q\chi_{S_i}]$. We now establish such a bound.

We begin with identifying cases in which both terms are equal to $0$,
and for that we use the orthogonality of the decomposition
$\{g^{=S}\}_{S\subset [n]}$. Afterwards, we only rely on the fact
that $g^{=S}$ depends only on the coordinates in $S$.

\begin{lem}\label{lem:upper bound on f=S}
  Let $q$ be some integer,  let $g\in L^2(\Omega, \nu)$, and let
  $S_1,\ldots, S_q \subset [n]$ be some sets.
  Suppose that some $j\in [n]$ belongs to exactly one of the sets $S_1,\ldots, S_q$. Then
  \[
    \mathbb{E}\left[\prod_{i=1}^q g^{=S_i}\right]=0 \quad \text{and} \quad
    \mathbb{E}\left[\prod_{i=1}^q\chi_{S_i}\right]=0.
    \]
  \end{lem}
  \begin{proof}
    Assume without loss of generality that $j\in S_1$.
    The second equality $\mathbb{E}\left[\prod_{i=1}^q\chi_{S_i}\right]=0$
    follows by taking expectation over $\chi_j$,
    using the independence between the random variables $\chi_i$.
    For the first equality, observe that the function
    $\prod^{q}_{i=2}g^{=S_i}$ depends only on coordinates in $S_2\cup
    \cdots, S_q \subset [n]\setminus \{j\}$. Hence the properties of
    the Efron--Stein decomposition imply
    \begin{align*}
      0 = \left\langle g^{=S_1}, \prod_{i=2}^q g^{=S_i}\right\rangle
      =\mathbb{E}\left[\prod_{i=1}^q g^{=S_i}\right]. & \qedhere
    \end{align*}
    \end{proof}

    Thus we only need to consider terms corresponding to $S_1,\ldots,S_q$
    in which each coordinate appears in at least two sets.
    To facilitate our inductive proof we work with general functions
    $f_i$ that depend only on coordinates of $S_i$
    (rather than only with the functions of the form $g^{=S_i}$).
    \begin{lem}\label{lem:tough upper bound}
      Let $f_1,\ldots, f_q\in L^2(\Omega, \nu)$ be functions that
      depend on sets $S_1,\ldots, S_q$ respectively. Let $T_i$ for
      $i=3,\ldots, q$ be the set of coordinates covered by the sets
      $S_1,\ldots, S_q$ exactly $i$ times. Then
      \[
        \left|\mathbb{E}\left[\prod_{i=1}^q f_i \right]\right|\le
        \prod_{i=1}^q \|f_i\|_2
        \cdot
        \prod_{j=3}^q\sigma_{T_j}^{2-j}.
        \]
      \end{lem}
      \begin{proof}
        The proof is by induction on $n$, simultaneously for all functions.
        We start with the case $n=1$, which we prove
        by reducing to the case that all $f_i$ are eqal.
        \subsection*{The case $n=1$.}
        Here each $f_i$ either depends on a single input
        or is constant and depends only on the empty set.
        We may assume that none of the $f_i$'s is constant,
        as otherwise we may eliminate it from the inequality
        by dividing by $|f_i|$.
        By the generalised H\"{o}lder inequality we have
        \[
         \left|\mathbb{E}\left[\prod_{i=1}^q f_i \right]\right| \le \prod_{i=1}^q\| f_i \|_q. \label{holder}
         \]
        Hence the case $n=1$ of the lemma will follow once
        we prove it assuming all the $f_i$ are equal.
        \subsection*{The $n=1$ case with equal $f_i$'s}
        We show that if $(\Omega, \nu)$ is a discrete probability
        space in which any atom has probability at least
        $p$, then  $\|f\|_q^q \le
        \|f\|_2^q\sigma^{2-q}$, where $\sigma=\sqrt{p(1-p)}$.

        While the inequality $\|f\|_2\le
        \|f\|_q$ holds in any probability space,
        the reverse inequality holds in any measure space
        where each atom has measure at least $1$.
        Accordingly, we consider the measure $\tilde{\nu}$
        on $\Omega$ defined by $\tilde{\nu}(x)=\nu(x) p^{-1}$. Then
        \[
          \|f\|_{q,\nu}^q=
        p\|f\|_{q,\tilde{\nu}}^q \le
        p\|f\|^q_{2, \tilde{\nu}} =
        p^{1-\frac{q}{2}}\|f\|_{2,\nu}^q \le
        \sigma^{2-q} \|f\|_{2,\nu}^{q}.
      \]
      This completes the proof of the $n=1$ case.
       \subsection*{The inductive step}
       Let $f_1,\ldots,f_q \in L^2(\Omega, \nu)$ be functions. Let
       $\mathbf{x} \sim \prod_{i=1}^{n-1} (\Omega _i,\nu_i)$. By the
       $n=1$ case we have:
       \[
        \left|\mathbb{E}\left[\prod_{i=1}^q f_i\right]\right|
        =\left|\mathbb{E}_{\mathbf{x}}
        \left[\mathbb{E}
        \left[\prod_{i=1}^q (f_i)_{[n-1]\to \mathbf{x}}
        \right]\right]\right|
        \le
        \mathbb{E}_{\mathbf{x}}
        \left[\prod_{i=1}^q\| (f_i)_{[n-1]\to
          \mathbf{x}}\|_2\sigma_n^{j}
          \right],
       \]
       where $j$ is $2-i$ if $n\in T_i$ for $i\geq 3$,
       and otherwise $0$.
       The lemma now follows by applying the inductive
       hypothesis on the functions
       $\mathbf{x}\rightarrow \|(f_i)_{[n-1]\to \mathbf{x}}\|$
       and using
       $\left\| \left\|(f_i)_{[n-1]\to \mathbf{x}} \right\|_2
       \right\|_{2,\mathbf{x}}= \|f_i\|_2$.
        \end{proof}

    \begin{proof}[Proof of Proposition \ref{prop:reduction}]
     We wish to upper bound
     \[
       \mathbb{E}[g^q]=\sum_{S_1,\ldots,S_q}\mathbb{E}\left[\prod_{i=1}^q
       g^{=S_i}\right]
     \]
     by
     \[
       \sum_{S_1,\ldots,S_q}\mathbb{E}\left[\prod_{i=1}^q
       \chi_{S_i}\right]\prod_{i=1}^q \|g^{=S_i}\|_2.
     \]
     We upper bound each term participating in the expansion of $g^q$
     by the corresponding term in $\tilde{g}^q$.
     In the case the sets $S_i$ cover some element exactly once,
     Lemma \ref{lem:upper bound on f=S} implies that both terms are
     $0$. Otherwise, the sets $S_i$ cover each element either $0$
     times or at least $2$ times; let $T_i$ be the set of elements of $S_1,\ldots,S_q$
     appearing in exactly $i$ of the sets (as in Lemma \ref{lem:tough upper bound}).
     By the  assumption of the proposition, we have $\mathbb{E}\left[\prod_{i=1}^q
     \chi_{S_i}\right]\ge \prod_{i=3}^q\sigma_{T_i}^{2-|T_i|}$.
     The proof is concluded by combining this with the upper bound
     on $\mathbb{E}\left[\prod_{i=1}^q g^{={S_i}}\right]$
     following from Lemma \ref{lem:tough upper bound} with $f_i = g^{=S_i}$.
   \end{proof}

   \begin{proof}[Proof of Theorem \ref{thm:es}]
 Let $\sigma_i'=\sqrt{p_i/4(1-p_i/4)}$.
 We choose $\chi_i$ to be the
 $\frac{p_i}{4}$-biased character,
 $\chi_i = \frac{x_i-p_i/4}{\sigma_i'}$.
 Clearly $\chi_i$ has mean $0$ and variance
 $1$, and a direct computation shows that
 $\mathbb{E}\left[\chi_i^j\right]
 \geq (\sigma_i)^{2-j}$ for all integer $j> 2$,
 hence all of the conditions of Proposition
 \ref{prop:reduction} hold.

 Denote $\sigma'_S=\prod_{i\in S}\sigma'_i$ and set $h=T_{\frac{1}{4}}f$, $g=\mathrm{T}_{\frac{1}{2q^{1.5}}} h$.
By Proposition \ref{prop:reduction} and Theorem
\ref{thm:qth moment with different norms} we have

\[
\| \mathrm{T}_{\frac{1}{8q^{1.5}  }   } f \|_q^q = \| g \|_q^q
\le \|\tilde{g} \|_q^q \le \sum_S (\sigma'_S)^{2-q} \|\mathrm{D}_S[\tilde{h}]\|_2.
\]
We note that by Parseval, the $2$-norm of $\tilde{h}$ and its derivatives are
equal to the $2$-norm of $h$ and its Laplacians,
and thus the last sum is equal to
\[
\sum_{S}(\sigma'_S)^{2-q} \| \mathrm{L}_S[h]\|_2^q \le
\sum_{S}(\sigma_S)^{2-q} \| \mathrm{L}_S[f]\|_2^q.
\]
In the last inequality we used $\sigma'_S\geq 2^{-|S|} \sigma_S$ and
$\| \mathrm{L}_S[h]\|^{q}\leq 2^{-q|S|} \| \mathrm{L}_S[f]\|_2^q$
(which follows from Parseval).
This completes the proof of the theorem.
\end{proof}

\section{An invariance principle (for global functions)}
\label{sec:inv}

Invariance (also known as Universality) is a fundamental paradigm
in Probability, describing the phenomenon that many random processes
converge to a specific distribution
that is the same for many different instances of the process.
The prototypical example is the Berry-Esseen Theorem,
giving a quantitative version of the Central Limit Theorem
(see e.g.\ \cite[Section 11.5]{o2014analysis}). More sophisticated instances of
the phenomenon that have been particularly influential on recent research
in several areas of Mathematics include the universality of
Wigner's semicircle law for random matrices (see \cite{mehta2004book})
and of Schramm--Loewner evolution (SLE)
e.g.\ in critical percolation (see \cite{smirnov2006survey}).

In the context of the cube, the Invariance Principle is a powerful tool
developed by Mossel, O'Donnell and Oleszkiewicz \cite{mossel2010noise}
while proving their `Majority is Stablest' Theorem,
which can be viewed as an isoperimetric theorem for the noise operator.
Roughly speaking, the result (in a more general form due to
Mossel \cite{mossel2010gaussian}) is that `majority functions'
(characteristic functions of Hamming balls) minimise noise sensitivity
among functions that are `far from being dictators'.
The Invariance Principle converts many problems on the cube
to equivalent problems in Gaussian Space;
in particular, `Majority is Stablest' is converted
into an isoperimetric problem in Gaussian Space
which was solved by a classical theorem
of Borell \cite{borell1985geometric}
(half-spaces are isoperimetric).

In the basic form (see \cite[Section 11.6]{o2014analysis})
of the Invariance Principle,
we consider a multilinear real-valued polynomial $f$ of degree $\le k$
and wish to compare $f(\xb)$ to $f(\yb)$,
where $\xb$ and $\yb$ are random vectors
each having independent coordinates,
according a smooth (to third order) test function $\phi$.
(Comparison of the cumulative distributions requires $\phi$ to be
a step function, but this can be handled by smooth approximation.)
The version of \cite[Remark 11.66]{o2014analysis}
 shows that if the coordinates $x_i$
have mean $0$, variance $1$ and are suitably hypercontractive
(satisfy $\|a + \rho b x_i\|_3 \le \|a + bx_i\|_2$
for any $a,b \in \mb{R}$), and similarly for $y_i$, then
\begin{equation} \label{eq:rem}
\big| \mb{E}[\phi(f(\xb))] - \mb{E}[\phi(f(\yb))] \big|
\le \tfrac{1}{3} \|\phi'''\|_\infty \rho^{-3k}
\sum_{i \in [n]} \mathrm{I}_i(f)^{3/2}.
\end{equation}

The hypercontractivity assumption applies e.g.\
if the coordinates are standard Gaussians or $p$-biased bits
(renormalised to have mean $0$ and variance $1$) with $p$ bounded away
from $0$ or $1$, but if $p=o(1)$ then we need $\rho=o(1)$,
in which case their theorem becomes ineffective.
We will apply our hypercontractivity inequality to obtain
an invariance principle that is effective for small
probabilities and functions with small generalised influences.
We adopt the following setup.

\begin{setup} \label{set:inv}
Let $\sigma_1,\ldots,\sigma_n>0$, let
$\mathbf{X}=(\mathbf{X}_1,\dots, \mathbf{X}_n)$ and
$\mathbf{Y}=(\mathbf{Y}_1,\ldots, \mathbf{Y}_n)$ be random vectors
with independent coordinates, where each
$X_i$ and $Y_i$ are real-valued random variable with mean $0$, variance $1$,
and satisfy $\|X_i\|_3^3 \le \sigma_i^{-1}$ and $\|Y_i\|_3^3\le \sigma_i^{-1}$.
Let $f \in \mb{R}[v]$ be a multilinear polynomial
of degree $d$ in $n$ variables $v=(v_1,\dots,v_n)$.
Let $\phi:\mb{R} \to \mb{R}$ be smooth.
\end{setup}

For $S \sub [n]$ we write $\hat{f}(S)$ for the
coefficient in $f$ of $v_S = \prod_{i \in S} v_i$.
We write $W_{S}(f)= \sum_{J: S \sub J} \hat{f}(J)^2$
and similarly to Section \ref{subsec:multilinear} we define the
generalised influences by
$\mathrm{I}_S(f) = W_{S}(f) \prod_{i \in S} \sS_{i}^{-2}$.

We write $\mathrm{T}_\rho[f]=\sum_{S\sub [n]}\rho^{|S|}\hat{f}(S)v_S$.

Now we state our invariance principle,
which compares $f(\mathbf{X})$ to $f(\mathbf{Y})$.

\begin{thm}\label{thm:inv}
Under Setup \ref{set:inv}, if $\mathrm{I}_S[f]\le \epsilon$ for each
nonempty set $S$, then
\[
  \left| \mb{E}[\phi(f(\mathbf{X}))] - \mb{E}[\phi(f(\mathbf{Y}))]
  \right|
  \le 2^{5d} \|\phi'''\|_\infty W_{\emptyset}(f) \sqrt{\epsilon}.
\]
\end{thm}

The term $W_{\emptyset}(f)$ can be replaced by either
$\mathbb{E}[f(\mathbf{X})^2]$ or
$\mathbb{E}[f(\mathbf{Y})^2]$ as they are all equal.

\skipi

Theorem \ref{thm:inv} can be informally interpreted as saying that if
a multilinear, low degree polynomial $f$ is global, then the distribution
of $f(\mathbf{X})$ does not really depend on the distribution of $\mathbf{X}$
except for the mean and variance of each coordinate.

In particular, it implies that plugging in the $p$-biased characters into
$f$ results in a fairly similar distribution to the one resulting from plugging in
the uniform characters into $f$. A posteriori, this may be seen as an intuitive explanation
for Theorem \ref{thm:Hypercontractivity}, as the standard hypercontractivity
theorem holds in the uniform cube.

\skipi
Next, we set up some notations and preliminary observations for the proof of Theorem
\ref{thm:inv}.
Throughout we fix
$\mathbf{X}$, $\mathbf{Y}$, $f$, and $\phi$ as in Setup \ref{set:inv}. We write
$\mathbf{X}_S=\prod_{i\in S}\mathbf{X}_i$, and similarly for $\mathbf{Y}$.
Recall that $f = \sum_S \hat{f}(S) v_S$ is a (formal) multilinear polynomial in $\mb{R}[v]$ of degree $d$.
Note that $f(\mathbf{X}) = \sum_S \hat{f}(S)\mathbf{X}_S$
has $\mb{E}[f(\mathbf{X})^2] = \sum_S \hat{f}(S)^2$,
as $\mb{E}\mathbf{X}_S^2 = 1$
and $\mb{E}[\mathbf{X}_S \mathbf{X}_T] = 0$ for $S \ne T$.
The random variable $f(\mathbf{X})$ has the orthogonal decomposition
$f = \sum_S f^{=S}$ with each $f^{=S} = \hat{f}(S)\mathbf{X}_S$.
Further note that $\mathrm{L}_S f(\mathbf{X}) = \sum_{J: S \sub J} \hat{f}(J)\mathbf{X}_J$
so we have the identities
\[
  \mathrm{I}_S(f) \prod_{i \in S} \sS_i^2 =
  \mb{E}[(\mathrm{L}_S f(\mathbf{X}))^2]=
  \mb{E}[(\mathrm{L}_S f(\mathbf{Y}))^2]=
  \sum_{J: S \sub J} \hat{f}(J)^2 = W^{S^\ua}(f).
  \]

We apply the replacement method
as in Section \ref{sec:hyp} (and as in the proof
of the original invariance principle by
Mossel, O'Donnell and Oleszkiewicz \cite{mossel2010noise}).
For $0 \le t \le n$, define
$\mathbf{Z}^{:t} = (\mathbf{Z}^{:t}_1,\dots,\mathbf{Z}^{:t}_n)
= (\mathbf{Y}_1,...,\mathbf{Y}_t,\mathbf{X}_{t+1},...,\mathbf{X}_n)$,
and note that $f(\mathbf{Z}^{:t})$ has the orthogonal decomposition
$f(\mathbf{Z}^{:t}) = \sum_S f(\mathbf{Z}^{:t})^{=S}$ with
\[
  f(\mathbf{Z}^{:t})^{=S} = \hat{f}(S) \mathbf{Z}_S = \hat{f}(S)
  \mathbf{Y}_{S\cap [t]} \mathbf{X}_{S\setminus [t]}.
  \]

\begin{proof}[Proof of Theorem \ref{thm:inv}]
We adapt the exposition in \cite[Section 11.6]{o2014analysis}.
As $\mathbf{Z}^{:0}=\mathbf{X}$ and $\mathbf{Z}^{:n}=\mathbf{Y}$ we have
by telescoping and the triangle inequality
\[
  |\mb{E}[\phi(f(\mathbf{X}))] - \mb{E}[\phi(f(\mathbf{Y}))] |
\le \sum_{t=1}^n
| \mb{E}[\phi(f(\mathbf{Z}^{:t-1}))] - \mb{E}[\phi(f(\mathbf{Z}^{:t}))] |.
\]
Consider any $t \in [n]$ and write
\[ f(\mathbf{Z}^{:t-1}) = U_t + \DD_t \mathbf{Y}_t
\ \ \text{ and } \ \
f(\mathbf{Z}^{:t}) = U_t + \DD_t \mathbf{X}_t,
\ \ \text{ where } \]
\[ U_t = \mathrm{E}_t f(\mathbf{Z}^{:t-1}) = \mathrm{E}_t f(\mathbf{Z}^{:t})
\ \ \text{ and } \ \
\DD_t = \mathrm{D}_t f(\mathbf{Z}^{:t-1}) = \mathrm{D}_t f(\mathbf{Z}^{:t}).
\]
Both of the functions $U_t$ and $\DD_t$ are independent of the random
variables $X_t$ and $Y_t$.

By Taylor's Theorem,
\begin{align*}
\phi(f(\mathbf{Z}^{:t-1})) & = \phi(U_t) + \phi'(U_t) \DD_t \mathbf{Y}_t
+ \tfrac{1}{2} \phi''(U_t) (\DD_t \mathbf{Y}_t)^2
+ \tfrac{1}{6} \phi'''(A) (\DD_t \mathbf{Y}_t)^3,
\ \ \text{ and }  \\
\phi(f(\mathbf{Z}^{:t})) & = \phi(U_t) + \phi'(U_t) \DD_t \mathbf{X}_t
+ \tfrac{1}{2} \phi''(U_t) (\DD_t \mathbf{X}_t)^2
+ \tfrac{1}{6} \phi'''(A') (\DD_t \mathbf{X}_t)^3,
\end{align*}
for some random variables $A$ and $A'$.
As $\mathbf{X}_t$ and $\mathbf{Y}_t$ have mean $0$ and variance $1$ we have
$0 = \mb{E}[\phi'(U_t) \DD_t \mathbf{Y}_t]
= \mb{E}[\phi'(U_t) \DD_t \mathbf{X}_t]$ and
$ \mb{E}[ \phi''(U_t) (\DD_t)^2 ]
= \mb{E}[ \phi''(U_t) (\DD_t \mathbf{Y}_t)^2 ]
= \mb{E}[ \phi''(U_t) (\DD_t \mathbf{X}_t)^2 ]$, so
\[
  | \mb{E}[\phi(f(\mathbf{Z}^{:t-1}))] - \mb{E}[\phi(f(\mathbf{Z}^{:t}))] |
\le \tfrac{1}{6} \|\phi'''\|_\infty
( \mb{E}[|\DD_t \mathbf{X}_t|^3] + \mb{E}[|\DD_t \mathbf{Z}_t|^3] )
\le \tfrac{1}{3} \|\phi'''\|_\infty\sigma_t^{-1} \|\DD_t \|_3^3.
\]

The function $\DD_t$ is the function $\mathrm{D}_t[f]$ applied on random
variables satisfying the hypothesis of Lemma \ref{lem:q-applying_hypercontractivity}. Moreover, $\mathrm{I}_S[\mathrm{D}_{t}[f]]$ is
either 0 when $t\in S$, or $\sigma_t^{2}\mathrm{I}_{S\cup\{t\}}[f]$ when $t\notin
S$, in which case $\mathrm{I}_S[f]\le \sigma_t^{2}\epsilon$. Hence, by
Lemma \ref{lem:q-applying_hypercontractivity} (with $q=3$), we obtain
\[
  \|\DD_t \|_3^3\le
6^{4.5 d}\sigma_t \sqrt{\epsilon}
\|\DD_t \|_2^2=
6^{4.5 d}\sigma_t \sqrt{\epsilon} \cdot \sum_{S\ni t}\hat{f}(S)^2.
\]

Hence,
\[
  \sum_{t=0}^n\tfrac{1}{3} \|\phi'''\|_\infty\sigma_t^{-1} \|\DD_t
  \|_3^3
  \le 6^{4.5 d}\sqrt{\epsilon} \tfrac{1}{3} \|\phi'''\|_\infty
  \sum_{S}|S|\hat{f}(S)^2
  \le 6^{4.5 d}\sqrt{\epsilon}\tfrac{d}{3} \|\phi'''\|_\infty
  W_{\emptyset}(f).
  \]
  This completes the proof of the theorem since
  $6^{4.5d}\frac{d}{3}\le 2^{12 d}$.
\end{proof}

\subsection{Applications of hypercontractivity}

We now list some corollaries of the invariance
principle. Following O'Donnell \cite[Chapter 11]{o2014analysis} one
can easily obtain the following variant of the `majority is stablest'
theorem of Mossel, O'Donnell and Oleszkiewicz \cite{mossel2010noise}
(see also \cite{mossel2010gaussian}).

The $p$-biased \emph{$\alpha$-Hamming ball} on $\{0,1\}^n$ is the function
$H_{\alpha}$ whose value is $1$ on an input $x$ if and only if $x$ has at least $t$ coordinates equal to $1$,
and $t$ is chosen so that $\mu_p(H_{\alpha})$ is as close to $\alpha$ as possible.
\begin{cor}\label{majority is stablest}
  For each $\epsilon>0$, there exists
  $\delta>0$, such that the following holds. Let $\rho\in (\epsilon,
  1-\epsilon)$, let $n>\delta^{-1}$, and let
  $f, g\in L^2(\{0,1\}^n,\mu_p)$. Suppose that $\mathrm{I}_S[f]\le
  \delta$ and that $\mathrm{I}_S[g]\le \delta$
  for each set $S$ of at most $\delta^{-1}$ coordinates. Then
  \[
    \left\langle \mathrm{T}_\rho{f},g \right\rangle \le \left
      \langle \mathrm{T}_\rho H_{\mu_p(f)},
      H_{\mu_p(g)}\right\rangle +\epsilon.
    \]
  \end{cor}
  The proof goes along the same lines of \cite{mossel2010gaussian}, and we omit it.
\skipi
As an additional application, one can obtain the following sharp threshold
result for \emph{almost} monotone Boolean functions. This statement
asserts that any such function which is global has a sharp
threshold. Let us remark that we have already established such a
result in the sparse regime (see Section \ref{sec:noise}). On the
other hand, the version below applies in the \emph{dense} regime.

With notation as in Section \ref{sec:noise}, we say that $f$ is
$(\delta,p,q)$-almost monotone if $p<q\in (0,1)$ and choosing
$\mathbf{x},\mathbf{y}\sim D(p,q)$ gives $\Pr[f(\mathbf{y})=0,
f(\mathbf{x})=1]<\delta$. We say that $f$ has an $\epsilon$-coarse
threshold in an interval $[p,q]$ if $\mu_p(f)>\epsilon$ and $\mu_q(f)<1-\epsilon$.

\begin{cor}
  For each $\epsilon>0$, there exists $\delta>0$, such that the
  following holds. Let $p<q<\frac{1}{2}$, and suppose that
  $q>(1+\epsilon)p$. Let $f$ be a $(\delta,p,q)$-almost monotone
  Boolean function having an $\epsilon$-coarse threshold in an
  interval $[p,q]$. Then there exists a set $S$ of size at most
  $\frac{1}{\delta}$, such that $\mathrm{I}_S[f]\ge \delta$ either with
  respect to the $p$-biased measure or with respect to the $q$-biased measure.
\end{cor}

The proof is similar to the one given by Lifshitz
\cite{lifshitz2018hypergraph}, so we only sketch it.

\begin{proof}[Proof sketch] First we observe that Corollary \ref{majority is
  stablest} extends to the one sided noise operator. Let $f_1=f$ be the
function viewed as a function on the $p$-biased cube, and let $f_2=f$
be the function viewed as a function on the $q$-biased cube. So assuming
for contradiction that $\mathrm{I}_S[f]\le \delta$ for each $S$, we
obtain an upper bound on
$\langle \mathrm{T}^{p\to q}f_1, f_2\rangle_{\mu_q}$ of the form
$\langle \mathrm{T}^{p\to q}H_{\mu_p(f)}, H_{\mu_q(f)}\rangle_{\mu_q}$

However, the $(\delta,p,q)$-almost monotonicity of $f$ implies the
lower bound
$\langle \mathrm{T}^{p\to q}f_1, f_2\rangle_{\mu_q} \rangle \ge
\mu_p(f)-\delta$.

Standard estimates on $\langle \mathrm{T}^{p\to q}H_{\mu_p(f)},
H_{\mu_q(f)}\rangle_{\mu_q}$ show that the lower bound and the upper
bound cannot coexist provided that $\delta$ is sufficiently small (see
\cite{lifshitz2018hypergraph}).
  \end{proof}

\section{Concluding remarks}

We are optimistic that our sharp threshold result
in the sparse regime will have many applications
in the same vein as the applications
of the classical sharp threshold results,
e.g.\ to Percolation \cite{benjamini2012sharp},
Complexity Theory \cite{friedgut1999sharp},
Coding Theory \cite{kudekar2017reed},
and Ramsey Theory \cite{friedgut2016sharp}.

In particular, it may be possible to estimate
the location of thresholds in the spirit of
the Kahn-Kalai conjecture \cite[Conjecture 2.1]{kahn2007thresholds}
that the threshold probability $p_c(H)$
for finding some graph $H$ in $G(n,p)$
should be within a log factor of
its `expectation threshold' $p_E(H)$
(the probability at which every subgraph $H'$ of $H$
we expect at least one copy of $H'$).
This question is interesting when $|V(H)|$
depends on $n$, e.g.\ if $H$ is a bounded degree spanning tree
it predicts $p_c(H) = O(n^{-1}\log n)$,
which was a longstanding open problem, recently
resolved by Montgomery \cite{montgomery2018trees}.

To obtain similar results from our sharp threshold theorem (Theorem \ref{thm:Sharp threshold result}),
one needs to show that the property of containing $H$
is not `local': writing $\mu_p = \mb{P}(H \sub G(n,p))$,
this means that if we plant any set $E$ of
$O(\log \mu_p^{-1})$ edges we still have
$\mb{P}(H \sub G(n,p) \mid E \sub G(n,p)) \le  \mu_p^{O(1)}$.
An open problem is to apply this approach to estimate
other thresholds that are currently unknown,
e.g.\ the threshold for containing any given $H$
of maximum degree $\DD$.

Our variant of the Kahn-Kalai conjecture on isoperimetric stability
is only effective in the $p$-biased setting for small $p$,
whereas the corresponding known results
\cite{keller2018dnf, keevash2018edgeiso}
for the uniform measure are substantial weaker.
This leaves our current state of knowledge
in a rather peculiar state, as in many related problems
the small $p$ case seems harder than the uniform case!
A natural open problem is give a unified approach
extending both results for all $p$.

Our final open problem
is to obtain a generalisation of Hatami's Theorem
to the sparse regime, i.e.\ to obtain a density increase
from $\mu_{p}\left(f\right) = o\left(1\right)$
to $\mu_{q}\left(f\right)\ge 1-\eps$
under some pseudorandomness condition on $f$;
we expect that a such result would have
profound consequences in Extremal Combinatorics.
\subsection*{Acknowledgment} We would like to thank Yuval Filmus, Ehud Friedgut, Gil Kalai, Nathan
Keller, Guy Kindler, and Muli Safra for various helpful comments and suggestions.

\bibliographystyle{plain}
\bibliography{refs}

\medskip

Peter Keevash (keevash@maths.ox.ac.uk),
Mathematical Institute, University of Oxford, UK.

\medskip

Noam Lifshitz (noamlifshitz@gmail.com),
Einstein Institute of Mathematics,
Hebrew University, Jerusalem, Israel.

\medskip

Eoin Long (long@maths.ox.ac.uk),
Mathematical Institute, University of Oxford, UK.

\medskip

Dor Minzer (minzer.dor@gmail.com),
Institute for Advanced Study, Princeton, NJ, USA.

\end{document}